%% file: main.tex
\documentclass[journal]{new-aiaa}
\usepackage[utf8]{inputenc}
\usepackage{textcomp}

\usepackage{algorithm}
\usepackage{algpseudocode}
\usepackage{subcaption}
\usepackage{graphicx}
\usepackage{amsmath}
\usepackage{booktabs}
\usepackage{empheq}
\usepackage{multirow}

\usepackage[version=4]{mhchem}
\usepackage{siunitx}
\usepackage{longtable,tabularx}
\usepackage{subfiles}
\graphicspath{{Figures/}}
\usepackage{amsthm}
\theoremstyle{plain}
\newtheorem{theorem}{Theorem}

\theoremstyle{remark}
\newtheorem{remark}{Remark}
\theoremstyle{definition}
\newtheorem{definition}{Definition}
\theoremstyle{problem}
\newtheorem{problem}{Problem} 
\hypersetup{
	hidelinks
}
\title{ABAMGuid+: An Enhanced Aerocapture Guidance Framework using Augmented Bank Angle Modulation}

\author{Kyle A. Sonandres \footnote{Graduate Research Assistant, Department of Aeronautics and Astronautics, Cambridge, MA, 02139, USA}}
\affil{Massachusetts Institute of Technology, Draper Scholar, Cambridge, MA, 02139}
\author{Thomas R. Palazzo\footnote{Senior Member of the Technical Staff, The Charles Stark Draper Laboratory, Inc., Cambridge, MA, 02139}}
\affil{The Charles Stark Draper Laboratory, Cambridge, MA, 02139, USA}
\author{Jonathan P. How\footnote{Richard C. Maclaurin Professor of Aeronautics and Astronautics, Department of Aeronautics and Astronautics, Cambridge, MA, 02139, USA}}
\affil{Massachusetts Institute of Technology, Cambridge, MA, 02139}

\begin{document}

\maketitle

\begin{abstract}
Aerocapture consists of converting a hyperbolic approach trajectory into a captured target orbit utilizing the aerodynamic forces generated via a single pass through the atmosphere. Aerocapture guidance systems must be robust to significant environmental variations and modeling uncertainty, particularly regarding atmospheric properties and delivery conditions. Recent work has shown that enabling control over both bank angle and angle of attack, a strategy referred to as augmented bank angle modulation (ABAM), can improve robustness to entry state and atmospheric uncertainties. In this work, we derive optimal control solutions for an aerocapture vehicle using ABAM. We first formulate the problem using a linear aerodynamic model and derive closed-form optimal control profiles using Pontryagin's Minimum Principle. To increase modeling fidelity, we also consider a quadratic aerodynamic model and obtain the solution directly using the optimality conditions. Both formulations are solved numerically using Gauss pseudospectral methods (via GPOPS, a software tool for pseudospectral optimal control), to validate the analytic solutions. We then introduce a novel aerocapture guidance algorithm, ABAMGuid+, which indirectly minimizes propellant usage by mimicking the structure of the optimal control solution, enabling efficient guidance by avoiding the complexity of solving the full optimal control problem online. Extensive Monte Carlo simulations of a Uranus aerocapture mission demonstrate that ABAMGuid+ increases capture success rates and reduces post-capture propellant requirements relative to previous methods. 
\end{abstract}

\section{Introduction}
\input{Introduction.tex}

\section{Augmented Bank Angle Modulation Aerocapture Problem Formulation}
\input{ABAM_Problem_Formulation.tex}

\section{Optimal Control Solutions} \label{sec: ocp solutions}
\input{Optimal_Control_Solutions.tex}

\section{Guidance Algorithm Design}
\input{Guidance_Algorithm_Design.tex}

\section{Guidance Performance in Varying Entry Scenarios} \label{sec: guidance performance} 
\input{Guidance_Performance.tex}

\section{Monte Carlo Simulations} \label{sec: mc results}
\input{Sim_Results.tex}

\section{Conclusions}
\input{Conclusion.tex}

\bibliography{references.bib}

\end{document}

%% file: Introduction.tex
\lettrine{A}{erocapture}
is an orbital insertion maneuver that converts a hyperbolic approach trajectory into a desired orbit via a single atmospheric pass. Following the atmospheric pass, a periapsis raise and an apoapsis correction, if necessary, are performed to enter the final target orbit. This is in contrast to traditional fully-propulsive insertion, in which a propulsive burn is used to enter a target orbit, and aerobraking, a  maneuver that begins with a traditional impulsive insertion into a highly elliptical orbit, followed by many passes through the atmosphere to lower the apoapsis altitude to the final target orbit. The aerocapture maneuver is illustrated in Fig. \ref{fig: aerocapture maneuver}. 

Eliminating reliance on an impulsive insertion burn has broad design implications. By controlling the aerodynamic accelerations appropriately during the atmospheric pass, the vehicle exits the atmosphere in an orbit that requires much less propellant consumption when compared to a traditional propulsive insertion maneuver. For missions to the ice giants, such as the Uranus Orbiter and Probe (UOP) flagship mission concept, aerocapture is an enabling technology. Baseline mission concepts utilizing traditional, fully-propulsive orbital insertion require a cruise time of 13-15 years and necessitate carrying large amounts of propellant ($\approx$ 1 km/s $\Delta V$) \cite{national2022origins}. By enabling higher cruise and arrival speeds, aerocapture can reduce flight time by 2-5 years while requiring less propellant, reducing launch mass and/or leaving more mass for the scientific payload \cite{restrepo2024mission}. For these reasons, NASA has been studying the feasibility of aerocapture at Uranus \cite{restrepo2024mission, shellabarger2024aerodynamic, scoggins2024aeroheating, matz2024analysis, dutta2024uranus}. The benefits of aerocapture are not unique to Uranus, and studies have been conducted at Mars \cite{wright2006mars,hamel2006improvement}, Neptune \cite{girija2020feasibility, lockwood2004neptune}, and Venus \cite{lockwood2006systems, girija2020feasibilityVenus}, among others, for similar reasons. 

Aerocapture is generally considered to introduce elevated risk, and guidance systems must be robust to large environmental and modeling uncertainties. The most widely studied aerocapture guidance methodology is bank angle modulation (BAM), in which the bank angle ($\sigma$) is used to rotate the lift vector around the free-stream velocity vector ($V_{\infty}$) while keeping the vehicle at a constant angle of attack ($\alpha$), lift ($L$), and drag ($D$). Such control architecture originates from early Apollo-era entry guidance algorithms \cite{moseley1969apollo}, \cite{graves1972apollo}. Over the past few decades, predictor-corrector algorithms have been widely adopted due to their ability to handle large dispersions while remaining computationally tractable for online applications \cite{lu2014entry}. PredGuid+A \cite{lafleur2011conditional} is a numerical predictor-corrector (NPC) method adapted from the skip-entry algorithm PredGuid \cite{putnam2008improving} specifically for aerocapture applications. More recently, FNPAG \cite{lu2015optimal} is a two-phase NPC algorithm designed to fly the optimal bang-bang bank angle profile, solving for the switching time in phase one and a constant bank angle in phase two to meet the apoapsis targeting constraint.

Bank angle modulation is typically selected due to its flight heritage \cite{moseley1969apollo, nelessen2019mars, way2007mars}. However, it imposes fundamental performance restrictions. By restricting the guidance system to a single control variable, the vehicle is under-actuated and is therefore increasingly susceptible to disturbances and modeling error. Recent interest has been placed in exploring alternative control strategies to potentially improve the vehicle's ability to respond to such uncertainty. One such approach is direct force control (DFC), an alternative control architecture that uses the angle of attack ($\alpha$) and side-slip angle ($\beta$) to change the lift, drag, and side-force ($Q$). An NPC algorithm for DFC is developed in \cite{deshmukh2020investigation} that modifies FNPAG for DFC application. A similar approach is deployed in \cite{matz2020development}, but it is found that the analytic solution does not exist without certain assumptions on the aerodynamics. The authors numerically find that the best solution is bang-bang in $\alpha$ and implement a two-phase NPC structure to command $\alpha$ while controlling $\beta$ with a PD controller to drive the orbit into a specified plane. Further numerical analysis is done in \cite{geiser2022optimal}, where the authors conclude that a bang-bang angle of attack solution is a good approximation to the optimal solution. 

Another approach, referred to as augmented bank angle modulation (ABAM) in our previous work \cite{sonandres2025aerocapture}, is a two-input control strategy that uses bank angle to control the direction of the lift vector, and angle of attack to modulate the lift and drag magnitudes. While several prior studies have explored the potential benefits of augmenting BAM with angle of attack for aerocapture or entry guidance \cite{harpold1978shuttle, lafleur2009angle, jits2004blended, starr2004aerocapture, jits2005closed}, these approaches have remained largely heuristic in nature. However, these works generally find that including angle of attack control can lead to performance gains. More recently, an embedded NPC strategy based on FNPAG is proposed for ABAM \cite{palazzo2025hybrid}, which is shown to significantly expand entry corridor width in comparison to FNPAG. Our previous work \cite{sonandres2025aerocapture} introduced a formalized framework for this approach, which is significantly improved in this work. We previously presented three main contributions. First, we derived closed form optimal angle of attack and bank angle trajectory solutions, assuming a linear aerodynamic model, showing that the result can be written in terms of two switching curves (the choice of which depends on $\sigma$) that predict when/ how the $\alpha$ switching sequence should occur. Second, we solved the problem numerically using GPOPS, validating the analytic solution subject to the linear assumption. Third, we presented ABAMGuid, a four-phase aerocapture guidance algorithm designed to fly the optimal control solution, and presented Monte Carlo analysis that demonstrated capture success rate and entry corridor width improvements in comparison to FNPAG. In this paper, we extend each contribution as follows:
\begin{enumerate}
	\item We formulate the ABAM optimal control problem using a higher-fidelity quadratic aerodynamic model, presenting a closed-form solution that describes the optimal $\sigma$ and $\alpha$ profiles as a function of state and costate variables. 
	\item The Gauss pseudospectral method is used to numerically solve the quadratic optimal control problem. The numerical results are shown to be consistent with the analytic theory, verifying the closed-form solution.
	\item A new four-phase guidance algorithm inspired by both optimal control solutions, ABAMGuid+, which improves upon ABAMGuid by performing unsaturated $\alpha$ and $\sigma$ control via continuous alpha-sigma modulation (CASM), a novel algorithm subprocess, is developed. In CASM, the two-variable optimization problem is reduced to a univariate root finding problem, enabling efficient two-channel control and improving the algorithm's ability to respond to dispersions with minimal added computational cost.
	\item Additional Monte Carlo analysis is conducted, considering effects of imperfect state information. In baseline entry scenarios, ABAMGuid+ leads to a 65\% reduction in 3-$\sigma$ $\Delta V$ and a 45\% reduction in 99-th \%-ile $\Delta V$ relative to ABAMGuid. For conservative entry flight path angle uncertainties, we observe up to a 50\% reduction in cases which fail to be captured into the desired orbit relative to FNPAG, and a 23\% improvement in mean $\Delta V$ compared to ABAMGuid. These results are supported by 144,000 Monte Carlo simulations, demonstrating the robustness of ABAMGuid+ across a variety of uncertainty conditions.
\end{enumerate} 

\begin{figure}[th!]
	\centering 
	\begin{minipage}[t!]{0.6\textwidth}
		\centering 
		\includegraphics[width=\textwidth]{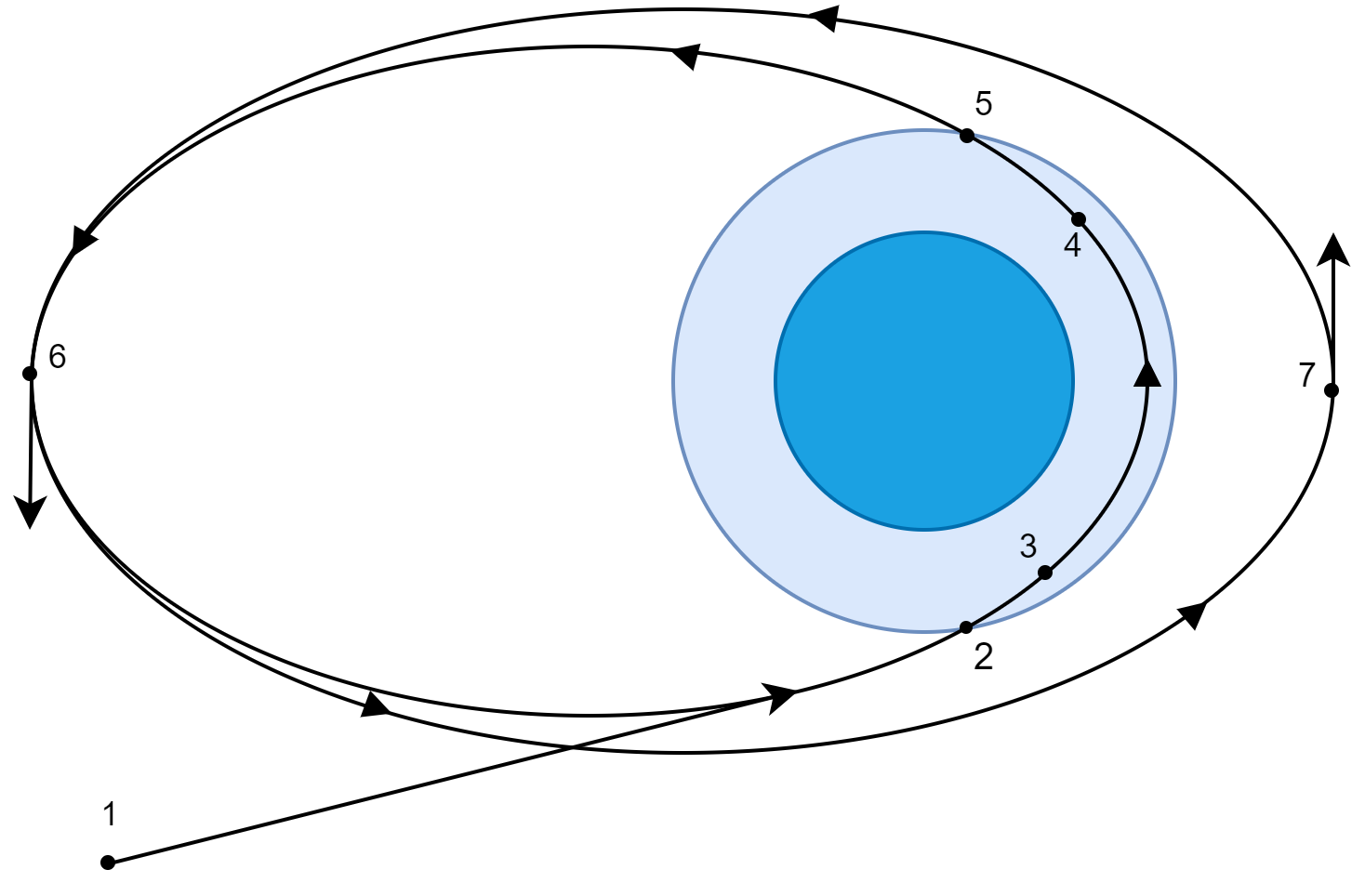}
		\label{fig: aerocapture figure}
	\end{minipage}%
	\hspace{0.05\textwidth}
	\begin{minipage}[t!]{0.29\textwidth}
		\small
		\centering 
		\begin{tabular}{|c|c|}
			\hline
			1 & Hyperbolic Approach Trajectory \\
			2 & Atmospheric Entry Interface \\
			3 & Active Guidance Begins \\
			4 & Active Guidance Ends \\
			5 & Atmospheric Exit \\
			6 & Periapsis Raise \\
			7 & Apoapsis Correction \\
			\hline
		\end{tabular}
	\end{minipage}
	\caption{Aerocapture maneuver. The spacecraft enters along a hyperbolic trajectory (1), entering the target planet's atmosphere at the atmospheric entry interface (2). The guidance system actively controls the vehicle once sufficient control authority is detected (3), continuing until the end of active guidance (4). The vehicle exits the atmosphere (5) in an elliptical orbit. Subsequent maneuvers, periapsis raise (6) and apoapsis correction (7), establish the desired orbit. } \label{fig: aerocapture maneuver}
\end{figure}

%% file: ABAM_Problem_Formulation.tex
In this section we define the optimal control problems that motivate the development of ABAMGuid+. Later sections analyze the solution structure and formulate an efficient four-phase guidance algorithm based on the structure.
\subsection{General ABAM Aerocapture Problem}
The three-dimensional equations of motion of a spacecraft inside the atmosphere of a rotating planet are \cite{miele1989optimal}
\begin{align}
	\dot{r} &= V \sin \gamma, \label{eq: rdotfull} \\
    \dot{\theta} &= \frac{V \cos \gamma \sin \psi}{r \cos \phi}, \\
    \dot{\phi} &= \frac{V \cos \gamma \cos \psi}{r}, \\
    \dot{V} &= -D(\alpha)-g_r \sin \gamma-g_\phi \cos \gamma \cos \psi +\Omega^2 r \cos \phi(\sin \gamma \cos \phi-\cos \gamma \sin \phi \cos \psi) , \\   	
    	\dot{\gamma} &= \frac{1}{V}\bigl[ L(\alpha) \cos \sigma+\left(V^2 / r-g_r\right) \cos \bigr. \gamma+g_\phi \sin \gamma \cos \psi +2 \Omega V \cos \phi \sin \psi  \\  
    		&\bigl.+\Omega^2  r \cos \phi(\cos \gamma \cos \phi v  +\sin \gamma \cos \psi \sin \phi) \bigr], \nonumber \\
    	\dot{\psi} &= \frac{1}{V}\bigl[ \frac{L(\alpha) \sin \sigma}{\cos \gamma}+\frac{V^2}{r} \cos \gamma \sin \psi \tan \phi+g_\phi \frac{\sin \psi}{\cos \gamma} 		\bigr. -2 \Omega V(\tan \gamma \cos \psi \cos \phi-\sin \phi) \label{eq: psidotfull} \\ 
    		& \bigl.+\frac{\Omega^2 r}{\cos \gamma} \sin \psi \sin \phi \cos \phi\bigr], \nonumber
\end{align}
where $r$ is the vehicle's radial distance from the center of the planet, $\theta$ is longitude, $\phi$ is latitude, $V$ is the vehicle's planet-relative velocity, $\gamma$ is the flight path angle of the velocity vector, and $\psi$ is the heading angle of the velocity vector. The planet's angular rotation rate is $\Omega$, and $g_r$ and $g_\theta$ are the radial and gravitational acceleration components, respectively, given as
\begin{equation}
\begin{gathered}
g_r=\frac{\mu}{r^2}\left[1+J_2\left(\frac{R_0}{r}\right)^2\left(1.5-4.5 \sin ^2 \phi\right)\right], \quad g_\phi=\frac{\mu}{r^2}\left[J_2\left(\frac{R_0}{r}\right)^2(3 \sin \phi \cos \phi)\right],
\end{gathered}
\end{equation}
where $J_2$ is the planet's zonal coefficient and $\mu$ is the gravitational parameter. The lift ($L$) and drag ($D$) accelerations are
\begin{equation} \label{eq: LD}
	L(\alpha) = \frac{1}{2m} \rho V^2 S C_L(\alpha), \quad
	D(\alpha) = \frac{1}{2m} \rho V^2 S C_D(\alpha),
\end{equation}
where atmospheric density is defined by $\rho$, $S$ is the vehicle's reference area, and $m$ is the vehicle's mass. The bank angle, $\sigma$, is the vehicle's rotation around the velocity vector. In traditional bank angle modulation (BAM) based aerocapture guidance, bank angle is taken as the sole control variable, while angle of attack, $\alpha$, is assumed to be fixed or prescribed as a function of velocity. In augmented bank angle modulation (ABAM), $\alpha$ and $\sigma$ are controlled simultaneously, allowing for greater control authority via manipulation of the aerodynamic lift and drag coefficients $C_L$ and $C_D$. In this work, we explicitly note the dependence of lift and drag on $\alpha$ to emphasize the additional control channel.  

Aerocapture is often formulated as an \textit{apoapsis-targeting problem}, where given an initial condition, we wish to find the bank angle and angle of attack control profiles that achieve a target apoapsis altitude for the post-atmospheric trajectory. This is given by the final state constraint
\begin{equation} \label{eq: apoapsis contraint}
    r_a(r_{\textrm{exit}}, V_{\textrm{exit}}, \gamma_{\textrm{exit}}) - r_a^* = 0,
\end{equation}
where $r_{\textrm{exit}}, V_{\textrm{exit}}, \gamma_{\textrm{exit}}$ are the radius, inertial velocity, and inertial flight path angle at atmospheric exit, respectively. In the aerocapture problem, $r_{\textrm{exit}}$ is typically given as a predefined fixed radius where atmospheric entry and exit are considered to take place. The target apoapsis radius is $r_a^*$, where $r_a$ is defined by Keplerian orbital mechanics as
\begin{equation} \label{eq: apoapsis targ}
r_a=a\left(1+\sqrt{1-\frac{V_{\textrm{exit}}^2 r_{\textrm{exit}}^2 \cos ^2\left(\gamma_{\textrm{exit}}\right)}{\mu a}}\right).
\end{equation}
The semi-major axis, $a$, is defined by 
\begin{equation} \label{eq: semimajoraxis}
    a = \frac{\mu}{2\mu / r_{\textrm{exit}} - V_{\textrm{exit}}^2}. \nonumber
\end{equation}
We wish to solve the \textit{optimal aerocapture problem}, which we define as the solution to the \textit{apoapsis targeting problem} which minimizes the propellant ($\Delta V$) required to insert the spacecraft into its final orbit. The total propellant consumption is computed as 
\begin{equation} \label{eq:dv}
\Delta V=\sqrt{2 \mu} \left( \left| \sqrt{\frac{1}{r_a}-\frac{1}{r_a+r_p^*}}-\sqrt{\vphantom{\frac{1}{r_a+r_p^*}}\frac{1}{r_a}-\frac{1}{2 a}}\right| + \left| \sqrt{\frac{1}{r_p^*}-\frac{1}{r_a^*+r_p^*}}-\sqrt{\frac{1}{r_p^*}-\frac{1}{r_a+r_p^*}}\right| \right),
\end{equation}
where $r_p^*$ is the target periapsis radius and $\mu$ is the planet gravitational parameter.

\subsection{Exit Velocity Targeting Reformulation}
A primary strength of ABAM is the ability to improve capture success rates in extreme entry conditions.
That is, when the entry flight path angle (EFPA) is excessively steep or shallow such that FNPAG (or any BAM-based algorithm) is unable to find a solution. In these cases, the predicted apoapsis altitude will often be nearly or completely undefined as the vehicle's trajectory is predicted to remain hyperbolic or crash into the planet, providing no meaningful feedback to the corrector to improve the solution in a NPC-based algorithm. Here, we pose an alternative objective function that eliminates the ill-conditioning of the apoapsis targeting criteria. The approach is to achieve a desired orbital energy, which can be approximated as achieving a desired exit velocity. For an elliptical orbit, the desired post-aerocapture orbital energy is given as
\begin{equation}
\epsilon^*=-\frac{\mu}{2 a^*}=\frac{V^2}{2}-\frac{\mu}{r}.
\end{equation}
Evaluating this at atmospheric exit, we can rewrite the desired exit velocity as
\begin{equation}
V_{\text {exit }}^*=\sqrt{2 \mu\left(\frac{1}{r_{\text {exit }}}-\frac{1}{a^*}\right)},
\end{equation}
where the desired semi-major axis, $a^*$, is 
\begin{equation}
a^*=\frac{r_a^*+r_p^*}{2}.
\end{equation}
When $r_a^* \gg r_p^*$ such that $r_a^* + r_p^* \approx r_a^*$, we will achieve the desired apoapsis radius with minimal error. The apoapsis targeting problem now becomes an exit velocity targeting problem where we seek to find the $\alpha$ and $\sigma$ control profiles that achieve the target exit velocity. It is important to note that this substitution is only acceptable for highly elliptical orbits as is the case considered here. The constraint is given as 
\begin{equation} \label{eq: V target}
V_{\text {exit }} - V^*_{\text {exit }} = 0.
\end{equation}
Importantly, $r_\text{exit}$ is less than $r_a^* + r_p^*$, so the quantity inside the square root is always positive thus avoiding imaginary solutions. Further, the objective function is singularity free in the case where the post aerocapture trajectory is hyperbolic. As will be seen in the following section, the choice of apoapsis targeting or velocity targeting does not alter the optimal control solution. However, in the guidance approach defined later in this paper, the guidance system will attempt to solve the velocity targeting problem in real time during every guidance cycle. We will ultimately seek to indirectly minimize $\Delta V$ by flying control trajectories based on the optimal control solution, which is a much more efficient guidance approach than solving the optimal control problem during each guidance call. In Sec. \ref{sec: ocp solutions} we develop the propellant optimal control profiles that will be the basis for the algorithm design. During testing and comparison of the algorithm, FNPAG is adapted to use the velocity targeting objective function for consistency.
\subsection{Longitudinal Problem Statement}
We observe that the objective function $\Delta V$ in Eq. \eqref{eq:dv} and terminal velocity targeting constraint in Eq. \eqref{eq: V target} depend only on the terminal values of the longitudinal motion variables $\mathbf{x_\text{lon}}$ = ($r, V, \gamma$), which are decoupled from the rest of the equations of motion. Therefore, we can isolate the longitudinal channel with the understanding that lateral guidance can be implemented via independent logic. We ignore terms of small magnitude due to planetary rotation and non-spherical gravity, as their effects are minor as long as the lateral motion variables do not appear in a terminal equality constraint \cite{lu2015optimal}. The simplified longitudinal dynamics are 
\begin{equation} \label{eq: longitudinal dynamics}
f(x_\textrm{lon}, u)=\left[\begin{array}{c}\dot{r} \\ \dot{V} \\ \dot{\gamma}\end{array}\right]=\left\{\begin{array}{l}V \sin \gamma \\ -D(\alpha)-\frac{\mu \sin \gamma}{r^2} \\ \frac{1}{V}\left[L(\alpha) u_1+\left(V^2-\frac{\mu}{r}\right) \frac{\cos \gamma}{r}\right] \end{array}\right.
\end{equation}
Note that in Eq. \eqref{eq: longitudinal dynamics}, $\cos \sigma$ has been replaced with $u_1$ to simplify notation. The optimal control problem to be solved in the following section is summarized in Problem \ref{problem: P1}.
\begin{problem} \label{problem: P1}
Find the $\sigma$ and $\alpha$ profiles which satisfy the terminal state constraint while minimizing fuel to correct the post-aerocapture orbit:
\begin{align}
&\underset{u \in \mathcal{U}}{\min} ~ J(x_\text{lon}(t_f)) = \Delta V  \\
 \text{s.t. \quad} &\dot{x}_\text{lon} = f(x_{lon}, u) \\
 &V_\text{exit} - V_\text{exit}^* = 0, \\
 &0^\circ \leq \sigma_\text{min} \leq \|  \sigma \| \leq \sigma_\text{max} \leq 180^\circ, \\
 &-30^\circ \leq \alpha_\text{min} \leq  \alpha \leq \alpha_\text{max} \leq 0^\circ, \\
 &L(\alpha) = f_L(\alpha), \quad D(\alpha) = f_D(\alpha)
\end{align}
\end{problem}
Here, $\sigma_\text{min}$ is larger than 0 degrees (typically about 15$^\circ$), and $\sigma_\text{max}$ is less than 180 degrees (typically about 105$^\circ$-165$^\circ$). These are set to ensure cross-range control authority. Angle of attack is similarly bounded to be within the vehicles allowable range, and is restricted to negative values due to near symmetry about zero and to require that direction of lift still be controlled by bank angle, as in BAM. These ranges define the admissible set of controls, $\mathcal{U}$. Note that in the general problem statement, lift and drag are nonlinear functions of $\alpha$, which will not permit a closed form solution. In Sec. \ref{subsec: aero fits}, we present the linear and quadratic aerodynamic models used in this work. In Problem \ref{problem: P2}, we solve the OCP using a linear aerodynamic model, and in Problem \ref{problem: P3}, we use the quadratic model.

\subsection{Aerodynamic Modeling} \label{subsec: aero fits}
Aerodynamic properties are modeled in simulation via CFD analysis in \cite{shellabarger2025aerodynamics} and stored in lookup tables as nonlinear functions of angle of attack. To obtain an analytic solution for the optimal control profiles, a functional approximation on the form of $C_D(\alpha)$ and $C_L(\alpha)$ is necessary. In this work, we consider linear and quadratic aerodynamic models.
The linear model assumes that the lift and drag coefficients vary linearly as a function of angle of attack as 
\begin{equation} \label{eq: linear aero model}
C_D(\alpha) = C_{D_\alpha} \alpha + C_{D,0}, \quad
C_L(\alpha) = C_{L_\alpha} \alpha + C_{L,0}.
\end{equation}
The quadratic model is given by  
\begin{equation}  \label{eq: quadratic aero model}
C_D(\alpha) = C_{D, \alpha^2} \alpha^2 + C_{D, \alpha} \alpha + C_{D,0}, \quad
C_L(\alpha) = C_{L, \alpha^2} \alpha^2 + C_{L, \alpha} \alpha + C_{L,0}.
\end{equation}
Figure \ref{Aero Fits} shows the linear (blue) and quadratic (black) drag models (left panel) resulting from fitting the nonlinear aerodynamic database used in simulation to a linear and quadratic function of $\alpha$, as well as the error percentage compared to truth (right panel). The nonlinear aerodynamic model is given in red for comparison. The coefficients are given in Table \ref{tab:aeroCoeffs}. In the right panel, the black dashed lines represent $\pm$ 3\% error for reference.
We observe that when $\alpha$ $\in$ [-10, -25], the linearly approximated $D$ error percentage is within 3\% for the entire range, suggesting that the linear model is a reasonable approximation of the true aerodynamics within this range. Previous work \cite{deshmukh2020investigation} has shown that this level of accuracy is well within the valid range for blunt bodied aerocapture vehicles. For the quadratic model over the same range, the error percentage drops to within $\approx$ 0.3\%, indicating that the quadratic model is extremely close to the true aerodynamics. 
The lift model follows a similar trend, and thus is omitted from the analysis.

\begin{figure}[h!]
	\centering
	\includegraphics[width=0.8\textwidth]{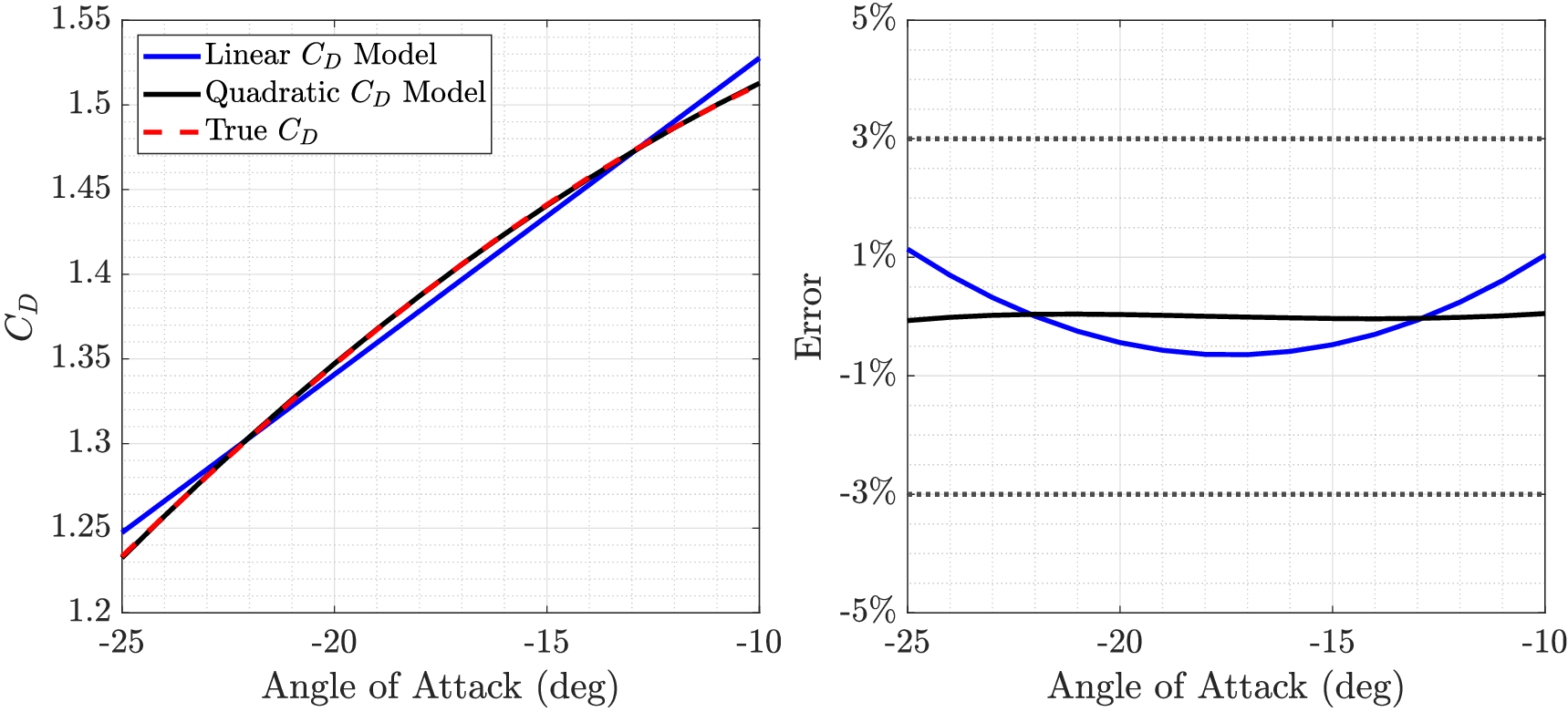}
	\caption{Linear (blue) and quadratic (black) $C_D$ approximations versus true aerodynamics (red) as a function of angle of attack (left panel). The right panel displays error percentage versus truth over the $\alpha$ range. The linear model well approximates the true nonlinear aerodynamics, and the quadratic model is nearly identical to truth.} \label{Aero Fits}
\end{figure}

\begin{table}[htbp] 
\centering
\caption{Aerodynamic coefficients for linear and quadratic models.}
\label{tab:aeroCoeffs}
\begin{tabular}{lcc}
\toprule
\toprule
\textbf{Coefficient} & \textbf{Linear Model} & \textbf{Quadratic Model} \\
\midrule
\(C_{D, 0}\) & 1.72 & 1.59 \\
\(C_{L, 0}\) & $7.07 \cdot 10^{-2}$ & $-2.71 \cdot 10^{-2}$ \\
\(C_{D, \alpha}\) & $1.87 \cdot 10^{-2}$ & $3.83 \cdot 10^{-3}$ \\
\(C_{L, \alpha}\) &$ -1.62 \cdot 10^{-2}$ &$ -2.82 \cdot 10^{-2}$ \\
\(C_{D, \alpha^2}\) & -- & $-4.25 \cdot 10^{-4} $\\
\(C_{L, \alpha^2}\) & -- & $-3.43 \cdot 10^{-4} $\\
\bottomrule
\bottomrule
\end{tabular}
\end{table}

%% file: Optimal_Control_Solutions.tex
This section presents optimal control profiles for an aerocapture vehicle using augmented bank angle modulation. 
The optimal control structure is derived using the linear aerodynamic model in Sec. \ref{sec: Linear OCP}, and the quadratic aerodynamic model in Sec.  \ref{sec: Quadratic OCP}. The OCP is solved numerically using GPOPS in Sec. \ref{sec: GPOPS}. 
Similar derivations can be found in Ref. \cite{lu2015optimal} for BAM and Ref. \cite{deshmukh2020investigation} for DFC. 
The objective of this exercise is to generate insight into the optimal control structure, which can then guide an efficient real-time implementation.
\subsection{Linear Aerodynamic Model} \label{sec: Linear OCP}
The following derivation is adapted from our previous work \cite{sonandres2025aerocapture}. Note that the solution is general to any targeting constraint that is dependent only on the terminal values of the longitudinal motion variables. 
\begin{problem} \label{problem: P2}
Find the $\sigma$ and $\alpha$ profiles that solve Problem \ref{problem: P1} subject to the linear aerodynamic models in Eq. \eqref{eq: linear aero model}.
\end{problem}
\begin{definition} \label{def: sigma bam}
Let  $\sigma^*_{\textrm{BAM}}: 
\mathbb{R} \rightarrow \{ \sigma_\textrm{min}, ~ \sigma_\textrm{max} \}$ be the optimal bank angle profile for the BAM aerocapture guidance problem \cite{lu2015optimal}, where $\sigma_\textrm{min}$ and $\sigma_\textrm{max}$ represent lift-up and lift-down flight, respectively. Define the bank angle switching curve as $H_\sigma = \lambda_\gamma$, where $\lambda_\gamma$ is the costate variable associated with the state variable $\gamma$. Then 
\begin{equation} \label{eq: sigma bam}
    \sigma_\textrm{BAM}^*= \begin{cases}\sigma_{\text {min }}, & \text { if } H_{\sigma}>0 \\ \sigma_{\text {max }}, & \text { if } H_{\sigma}<0\end{cases}
\end{equation}
The singular case where $\sigma_\textrm{BAM}^* \in \left[ \sigma_\textrm{min}, ~ \sigma_\textrm{max} \right]$ requires $H_\sigma = 0 $ in $ \left[t_1, ~ t_2 \right] \subset \left[t_0, ~t_f \right]$, which is shown to be impossible in Ref. \cite{lu2015optimal}. The switching time $t_\sigma$ refers to the time at which the switching curve crosses zero: $H_\sigma(t_\sigma) = 0$. 
\end{definition}
\begin{definition} \label{def: alpha switching functions}
For the two-phase flight given in Definition \ref{def: sigma bam}, define the switching curve associated with $\alpha$ during lift-up flight ($\sigma = \sigma_\textrm{min}$), $H_{\alpha, \textrm{up}}: \left[t_0,~ t_f \right] \rightarrow \mathbb{R}$, as
\begin{equation} \label{eq: H alpha up}
H_{\alpha, \textrm{up}}  = - \lambda_V C_{D, \alpha} + \lambda_{\gamma} \frac{C_{L,\alpha}}{V},
\end{equation}
where $\lambda_V$ is the costate variable associated with $V$. During lift-down flight ($\sigma = \sigma_\textrm{max}$), define $H_{\alpha, \textrm{down}}: \left[t_0,~ t_f \right] \rightarrow \mathbb{R}$, as
\begin{equation} \label{eq: H alpha down}
H_{\alpha, \textrm{down}}  = - \lambda_V C_{D, \alpha} - \lambda_{\gamma} \frac{C_{L,\alpha}}{V}.
\end{equation}
The switching time $t_{\alpha, \textrm{up}}$ refers to the time at which $H_{\alpha, \textrm{up}}(t) = 0$, and $t_{\alpha, \textrm{down}}$ refers to the time at which $H_{\alpha, \textrm{down}}(t) = 0$.
\end{definition}
\begin{theorem} \label{theorem: theorem 1}
Assume that the optimal bank angle for the two-input case (ABAM) is given by $\sigma^* = \sigma^*_\text{BAM}$, as in Definition \ref{def: sigma bam}. Then the optimal angle of attack profile, $\alpha^*: 
\mathbb{R} \rightarrow \{ \alpha_\textrm{min}, ~ \alpha_\textrm{max} \}$, is determined by 
\begin{equation} \label{eq: optimal aoa}
    \alpha^*= \begin{cases}\alpha_{\text {min }}, & \text { if } H_{\alpha}>0 \\ \alpha_{\text {max }}, & \text { if } H_{\alpha}<0\end{cases}, 
\end{equation}
where $H_\alpha = H_{\alpha, \textrm{up}}$ during lift-up flight, and $H_\alpha = H_{\alpha, \textrm{down}}$ during lift-down flight, given in Definition \ref{def: alpha switching functions}. 
\end{theorem}

\begin{proof}
We begin by forming the Hamiltonian for the two-input problem. There are no running costs, so it is formed by adjoining the costate variables ($\lambda_r$, $\lambda_V$, $\lambda_\gamma$) to the dynamics
\begin{equation} \label{eq: H}
    H(x, u, \lambda) = \lambda_r \dot{r} + \lambda_V \dot{V} + \lambda_\gamma \dot{\gamma}, 
\end{equation}
where the control vector is $u = \left[ u_1, ~ \alpha \right]$. By Pontryagin's Minimum Principle, the optimal control profile is given by 
\begin{equation}
u^* 
=
\begin{bmatrix}
u_1^* \\
\alpha^*
\end{bmatrix} 
=
\underset{u \in \mathcal{U}}{\arg \min} \ H \left(x, u, \lambda \right),
\end{equation}
We assume that $\sigma^* = \sigma^*_{\textrm{BAM}}$, as given in Definition \ref{def: sigma bam}, so we set $u_1^* = \cos \left( \sigma^*_{\textrm{BAM}} \right)$. For lift-up flight, $\sigma^*_{\textrm{BAM}} = \sigma_\textrm{min}$, so $u_1^* \approx 1$, and for lift-down flight, $\sigma^*_{\textrm{BAM}} = \sigma_\textrm{max}$, so $u_1^* \approx 1$. The solution will therefore produce $\alpha^*$.
Substituting the dynamics from Eq. \eqref{eq: longitudinal dynamics} into the Hamiltonian yields
\begin{equation} \label{eq: H with dynamics}
    H = \lambda_r \left(V \sin \gamma\right) + \lambda_V \left( -D(\alpha)-\frac{\mu \sin \gamma}{r^2}\right) + \lambda_\gamma \left(\frac{1}{V}\left[L(\alpha) u_1 +\left(V^2-\frac{\mu}{r}\right) \frac{\cos \gamma}{r}\right]\right),
\end{equation}
and when $u_1^*$ is inserted, there only remains a linear control dependence on $\alpha$ in $H$. We can thus rearrange the Hamiltonian as 
\begin{equation}
    H = H_0 + H_{\alpha, \textrm{(up/down)}} \alpha,
\end{equation}
where $H_{\alpha,  \textrm{up/down}}$ are the switching curves per Definition \ref{def: alpha switching functions}, and $H_0$ consists of terms that do not depend on the control variable. We can omit $H_0$, as it has no implications for the solution, and substitute in Eq. \eqref{eq: LD} for $L(\alpha)$ and $D(\alpha)$ to get an augmented Hamiltonian $\tilde{H}$
\begin{align} \label{eq: H nonlinear}
    \tilde{H} &= \left[-\lambda_V \left( q S C_{D, \alpha}\right)  \pm \lambda_\gamma \left( \frac{q}{V} S C_{L,\alpha}\right) \right] \alpha, \\
    &= q S \left( -\lambda_V C_{D, \alpha} \pm \lambda_\gamma \frac{C_{L, \alpha}}{V}\right) \alpha, \\
    &= H_{\alpha, \textrm{(up/ down)}} \alpha.
\end{align}
where $q = \frac{1}{2} \rho V^2$ is the dynamic pressure. Note that minimizing $\tilde{H}$ is equivalent to minimizing $H$. We can drop $q$ and $S$ since they do not affect the zero crossing, to get the final form of the switching curves
\begin{align}
    H_{\alpha, \textrm{up}}  &= - \lambda_V C_{D, \alpha} + \lambda_{\gamma} \frac{C_{L,\alpha}}{V}, \\
    H_{\alpha, \textrm{down}}  &= - \lambda_V C_{D, \alpha} - \lambda_{\gamma} \frac{C_{L,\alpha}}{V}. 
\end{align}
Since $\tilde{H}$ is affine in $\alpha$, the $\alpha$ that minimizes $\tilde{H}$ lies on the boundary of $\mathcal{U}$, yielding
\begin{equation} 
    \alpha^*= \begin{cases}\alpha_{\text {min }}, & \text { if } H_{\alpha}>0, \\ \alpha_{\text {max }}, & \text { if } H_{\alpha}<0,\end{cases}
\end{equation}
where $H_\alpha = H_{\alpha, \textrm{up}}$ during lift-up flight, and $ H_\alpha = H_{\alpha, \textrm{down}}$ during lift-down flight.
\end{proof}
\vspace{-2mm}
\begin{remark}
The control in the case where $ H_{\alpha,\textrm{up}}$ and $ H_{\alpha,\textrm{down}}$ equal zero for a non-zero time interval represents a singular arc, and it has been shown to be impossible for similar problems \cite{lu2015optimal}, \cite{deshmukh2020investigation}, but was not investigated in this work. 
\end{remark}
\begin{remark} \label{remark: bank switch}
Given the assumption that the optimal $\sigma$ profile follows the results in \cite{lu2015optimal}, we expect to see
\begin{equation} \label{eq: sig}
    \sigma^*= \begin{cases}\sigma_{\text {min }}, & \text { if } H_\sigma >0 \\ \sigma_{\text {max }}, & \text { if } H_\sigma<0 \end{cases}
\end{equation}
Notice that this solution is recovered at the intersection of the lift-up and lift-down switching functions, where $H_{\alpha,\textrm{up}} = H_{\alpha,\textrm{down}}$. We can therefore define the bank angle switching ($H_\sigma$) function accordingly, dropping the constant multiplication factor
\begin{equation}
    H_\sigma = H_{\alpha,\textrm{up}} - H_{\alpha,\textrm{down}} = \lambda_\gamma
\end{equation}
\end{remark}
\begin{remark} \label{remark 3}
We assume that the control profile consists of at most three switches ($\sigma$ may switch once and $\alpha$ may switch once both before and after the $\sigma$ switch). The following numerical analysis supports this assumption. Considering a nominal scenario where the vehicle begins flying at $\sigma_{\textrm{min}}$ and $\alpha_{\textrm{min}}$, the first switching time occurs when $H_{\alpha, \textrm{up}}$ crosses zero and corresponds to an $\alpha$ switch. The second occurs when $H_{\alpha, \textrm{up}} = H_{\alpha, \textrm{down}}$ and corresponds to a $\sigma$ switch. The third occurs when $H_{\alpha, \textrm{down}}$ crosses zero and corresponds to the second $\alpha$ switch. 

While we assume the vehicle begins flying $\sigma_{\textrm{min}}$ and $\alpha_\textrm{min}$ for consistency with FNPAG, it remains possible for other permutations of control sequences to exist. 
For example, a full-lift down flight is possible, and may be thought of as a special case of the bang-bang control scheme, where the duration of the $\sigma = \sigma_\textrm{min}$ phase is zero. In this case, $H_{\alpha, \textrm{down}}$ governs a potential $\alpha$ switch from $\alpha_\textrm{min}$ to $\alpha_\textrm{max}$. However, it is also possible for the duration of the $\alpha = \alpha_\textrm{max}$ phase to be zero. This would result in a trajectory where no control variable switches occur. 
For these reasons, the three-switch control profile is not an inherently limiting architecture.  
\end{remark}

\subsection{Quadratic Aerodynamic Model} \label{sec: Quadratic OCP}
\begin{problem} \label{problem: P3}
Find the $\sigma$ and $\alpha$ profiles that solve Problem \ref{problem: P1} subject to the quadratic aerodynamic models in Eq. \eqref{eq: quadratic aero model}.
\end{problem}
\begin{definition} \label{def: alpha curves}
For the two-phase flight given by Definition \ref{def: sigma bam}, let  $H_{\alpha, \textrm{up}}$ and  $H_{\alpha, \textrm{down}}$ be functions associated with $\alpha$ during lift-up and lift-down flight, as in Definition \ref{def: alpha switching functions}. To clarify notation, we no longer refer to these as switching curves since the Hamiltonian will no longer be affine in $\alpha$, resulting in a solution that is not bang-bang in $\alpha$. Define $H_{\alpha^2, \textrm{up}}: \left[t_0,~ t_f \right] \rightarrow \mathbb{R}$, the function associated with $\alpha^2$ during lift-up flight, as
\begin{equation} \label{eq: alpha2 up}
H_{\alpha^2, \textrm{up}} = \left( - \lambda_V C_{D, \alpha^2} + \lambda_{\gamma} \frac{C_{L,\alpha^2}}{V}\right), 
\end{equation}
and let $H_{\alpha^2, \textrm{down}}: \left[t_0,~ t_f \right] \rightarrow \mathbb{R}$, the function associated with $\alpha^2$ during lift-down flight, be
\begin{equation} \label{eq: alpha2 down}
H_{\alpha^2, \textrm{down}} = \left( - \lambda_V C_{D, \alpha^2} - \lambda_{\gamma} \frac{C_{L,\alpha^2}}{V}\right), 
\end{equation}
\end{definition}

\begin{definition} \label{def: optimal alpha curves}
Given $H_{\alpha,\textrm{up}}$, $H_{\alpha^2,\textrm{up}}$, $H_{\alpha,\textrm{down}}$, and $H_{\alpha^2,\textrm{down}}$ in Definition \ref{def: alpha curves}, define the optimal $\alpha$ control profile associated with $\alpha$ during lift-up flight, $\mathcal{A}_\textrm{up}: \left[t_0,~ t_f \right] \rightarrow \mathbb{R}$, as
\begin{equation}
    \mathcal{A}_\textrm{up} = \frac{-H_{\alpha,\textrm{up}}}{2  H_{\alpha^2,\textrm{up}}}.
\end{equation}
During lift-down flight, define  $\mathcal{A}_\textrm{down}: \left[t_0,~ t_f \right] \rightarrow \mathbb{R}$, as
\begin{equation}
    \mathcal{A}_\textrm{down} = \frac{-H_{\alpha,\textrm{down}}}{2  H_{\alpha^2,\textrm{down}}}.
\end{equation}
Both $H_{\alpha^2, \textrm{up}}$ and $H_{\alpha^2, \textrm{down}}$ are assumed to be nonzero to avoid singularities in the definitions of $\mathcal{A}_\textrm{up}$ and $\mathcal{A}_\textrm{down}$.
\end{definition}

\begin{theorem} \label{theorem: theorem 2}
Assume that the optimal bank angle for the two-input case (ABAM) is given by $\sigma^* = \sigma^*_\text{BAM}$, as in Definition \ref{def: sigma bam}. Then the optimal angle of attack profile, $\alpha^*: 
\left[t_0,~ t_f \right] \rightarrow \left[ \alpha_\textrm{min}, ~ \alpha_\textrm{max} \right]$, is determined by 
\begin{equation} \label{eq: quadratic alpha star}
    \alpha^*= \begin{cases}\alpha_{\text {min }} & \text { if } \mathcal{A}<\alpha_{\text {min }} \\ \mathcal{A} & \text { if } \quad \alpha_{\text {max }} \leq \mathcal{A} \leq \alpha_{\text {min }} \\ \alpha_{\text {max }} & \mathcal{A}>\alpha_{\text {max }}\end{cases}
\end{equation}
where $\mathcal{A} = \mathcal{A}_\textrm{up}$ during lift-up flight, and $\mathcal{A} = \mathcal{A}_\textrm{down}$ during lift-down flight, per Definition \ref{def: optimal alpha curves}. 
\end{theorem}

\begin{proof}
The core machinery of this proof is similar to the linear case. We begin by forming the Hamiltonian, as done in Eq. \eqref{eq: H}. Since $H$ is not affine in $\alpha$, we no longer rely on the minimum principle to derive $u^*$. Instead, the optimal control profile $u^*$ in the unconstrained space is now given by the necessary condition for optimality,
\begin{equation}
\frac{\partial H}{\partial u} = 0.
\end{equation}
Like before, the dynamics are substituted into the Hamiltonian and $u_1^* = \cos \left( \sigma^*_\textrm{BAM} \right)$ is inserted. Upon substitution, the Hamiltonian becomes a function of $\alpha$ and $\alpha^2$, so it can be rearranged as
\begin{equation}
    H = H_0 + H_{\alpha, \textrm{(up/down)}} \alpha + H_{\alpha^2, \textrm{(up/down)}} \alpha^2,
\end{equation}
where $H_0$ consists of terms that do not depend on the control variable, and $H_{\alpha, \textrm{(up/down)}}$ and $H_{\alpha^2, \textrm{(up/down)}}$ are the functions associated with $\alpha$ and $\alpha^2$, as given by Definition \ref{def: alpha curves}. 
We can omit $H_0$ since there is no control dependence and substitute Eq. \eqref{eq: LD} for $L(\alpha)$ and $D(\alpha)$ to get 
\begin{align}
    \Tilde{H} &= qS \left(  - \lambda_V C_{D, \alpha} \pm \lambda_{\gamma} \frac{C_{L,\alpha}}{V}\right) \alpha +  qS \left(  - \lambda_V C_{D, \alpha^2} \pm \lambda_{\gamma} \frac{C_{L,\alpha^2}}{V}\right) \alpha^2, \\
    &= \left(  - \lambda_V C_{D, \alpha} \pm \lambda_{\gamma} \frac{C_{L,\alpha}}{V}\right) \alpha +  \left(  - \lambda_V C_{D, \alpha^2} \pm \lambda_{\gamma} \frac{C_{L,\alpha^2}}{V}\right) \alpha^2, \\
    &= H_{\alpha, \textrm{(up/down)}} \alpha + H_{\alpha^2, \textrm{(up/down)}} \alpha^2,
\end{align}
Note that $q$ and $S$ can be dropped since they will cancel in the final solution. Since the control variable is no longer linear in the Hamiltonian, we can find $\alpha^*$ directly by leveraging the necessary condition to minimize the augmented Hamiltonian
\begin{equation}
    \frac{\partial\Tilde{H}}{\partial \alpha} =  H_{\alpha, \textrm{(up/down)}} + 2 H_{\alpha^2, \textrm{(up/down)}} \alpha^* = 0.
\end{equation}
Solving for $\alpha^*$ gives
\begin{equation}
\alpha^* = \frac{-H_{\alpha, \textrm{(up/down)}}}{2 H_{\alpha^2, \textrm{(up/down)}}} = \mathcal{A}_\textrm{(up/down)}
\end{equation}
Finally, we can impose the control magnitude constraints to arrive at the final optimal control solution
\begin{equation}
    \alpha^*= \begin{cases}\alpha_{\text {min }} & \text { if } \mathcal{A}_{\text{(up/down)}}<\alpha_{\text {min }} \\ \mathcal{A}_{\text{(up/down)}} & \text { if } \quad \alpha_{\text {max }} \leq \mathcal{A}_{\text{(up/down)}} \leq \alpha_{\text {min }} \\ \alpha_{\text {max }} & \mathcal{A}_{\text{(up/down)}}>\alpha_{\text {max }}\end{cases}
\end{equation}

\end{proof}
\begin{remark}
As in the linear case, we expect to recover information about the bank angle switch in some way due to the core assumption that $\sigma^* = \sigma_\text{BAM}^*$. We observe that the intersection of $\mathcal{A}_\text{up}$ and $\mathcal{A}_\text{down}$ corresponds to the intersection of $H_{\alpha, \textrm{up}}$ and $H_{\alpha, \textrm{down}}$, which was established in the linear case to be equivalent to the bank angle switching condition.  
\end{remark}
\begin{remark}
The angle of attack profile derived using quadratic aerodynamics is not bang-bang in nature, and parameterizing Eq. \eqref{eq: quadratic alpha star} for use in a guidance algorithm is a non-trivial task. Because of the added complexity associated with directly using this solution, in this work we only intend to take advantage of the finding that an unsaturated $\alpha$ flight regime exists, and that performance can be improved by incorporating logic to find an unsaturated $\alpha$-$\sigma$ control solution. This is done in continuous alpha-sigma modulation (CASM), the new subprocess presented in Sec. \ref{sec: casm}.
\end{remark}

\subsection{Numerical Verification and Comparison} \label{sec: GPOPS}
In order to verify the theory presented above, we solve Problem \ref{problem: P2} and Problem \ref{problem: P3} numerically using the Gauss Pseudospectral Optimization Software (GPOPS) \cite{benson2005gauss, rao2010algorithm}. The theorems presented above were not included when encoding the problem, so any similarities between the numerical results and control theory have arisen organically. The longitudinal dynamics used are as in Eq. \eqref{eq: longitudinal dynamics}. We enforce a final state constraint according to the apoapsis targeting criteria in Eq. \eqref{eq: apoapsis contraint}. Note that the velocity targeting constraint can be used in place of this, which will yield the same result. For convergence improvements, a new control variable associated with $\sigma$ is introduced $u = \left[ u_1, \quad u_2 \right] = \left[ \cos \sigma, \quad \sin \sigma \right]$. Note that this introduces a path constraint $u_1^2 + u_2^2 = 1$. The $\sigma$ magnitude constraint therefore becomes $u_{1, \textrm{min}} \leq u_1 \leq u_{1, \textrm{max}}$, where $u_{1,\textrm{max}} = \cos \sigma_{\textrm{min}}$ and $u_{1, \textrm{min}} = \cos \sigma_{\textrm{max}}$. Similarly, we enforce $u_{2, \textrm{min}} \leq u_2 \leq u_{2, \textrm{max}}$, where $u_{2,\textrm{min}} = \sin \sigma_{\textrm{min}}$ and $u_{2, \textrm{max}} = \sin \sigma_{\textrm{max}}$. A similar approach is taken in Ref. \cite{han2019rapid}. Bank angle is bounded between 15 and 165 degrees, and angle of attack is constrained between -10 and -25 degrees. This range is selected to maintain accuracy of the aerodynamic models, consistent with the analysis done in Sec. \ref{subsec: aero fits}. The cost to be minimized is $\Delta V$. Because we enforce the apoapsis targeting constraint, we eliminate the need for an apoapsis correction burn. Thus, the $\Delta V$ cost implemented here consists only of the first term of Eq. \eqref{eq:dv}. This combination was found to exhibit the strongest convergence properties. Linear aerodynamic models are implemented according to Eq. \eqref{eq: linear aero model}, and quadratic models are implemented according to Eq. \eqref{eq: quadratic aero model} using the parameters given in Table \ref{tab:aeroCoeffs}. Initial guesses are obtained by propagating the dynamics given the initial state, using a constant $\alpha$ profile, and a bang-bang $\sigma$ profile with a user-supplied $t_\sigma$. The atmospheric density profile is constructed via polynomial fit to the natural logarithm of the Uranus-GRAM nominal density profile. 

Two GPOPS solutions are presented, one using the linear aerodynamic model (Fig. \ref{fig: linear gpops}), and one using the quadratic model (Fig. \ref{fig: quadratic gpops}). Both cases use the nominal entry condition, such that $r_0$ = 1000 km, $\gamma_0$ = -10.79 degrees, and $V_0$ = 24.9 km/s. Once the solutions are found, the state and costate information is extracted to plot the analytic switching curves (linear case) and control solution (quadratic case) alongside the optimal control $\alpha$ and $\sigma$ trajectories found by GPOPS. 

Let us first consider the \textit{linear} case (Fig. \ref{fig: linear gpops}). Here, the optimal control solution exhibits a bang-bang $\alpha$ and $\sigma$ structure, where switches are dictated by the switching curves $H_{\alpha, \textrm{up}}$ and $H_{\alpha, \textrm{down}}$. The vehicle begins flying lift-up, with $u_1 \approx 1$. Therefore, we expect $H_{\alpha, \textrm{up}}$ to govern the first $\alpha$ switch. $H_{\alpha, \textrm{up}}$ crosses zero at approximately $t = 265$ seconds, and $\alpha$ switches from $\alpha_{\textrm{min}}$ to $\alpha_{\textrm{max}}$ at about the same time. We next expect a $\sigma$ switch at the second switching event. The point at which $H_\sigma$ crosses zero is equivalent to the switching curve intersection $H_{\alpha, \textrm{up}} = H_{\alpha, \textrm{down}}$, which occurs near $t = 271$ seconds. We see that $\sigma$ switches from $\sigma_{\textrm{min}}$ to $\sigma_{\textrm{max}}$ ($u_1 \approx -1$) at this time. Because the vehicle is now flying lift-down, $u_1 \approx -1$, and we expect a second $\alpha$ switch when $H_{\alpha, \textrm{down}}$ crosses zero. This occurs at $t \approx 276$, and the corresponding $\alpha$ switch happens at that time.    

In the \textit{quadratic} case (Fig. \ref{fig: quadratic gpops}), the optimal control solution defines $\alpha^*$ directly via $\alpha^*$ = $\mathcal{A}_{\textrm{up/down}}$, when $\mathcal{A}_{\textrm{up/down}}$ is within control magnitude constraints. The vehicle begins flying lift-up, so we expect the control to be defined by $\mathcal{A}_\textrm{up}$. Because the curve exceeds $\alpha_\textrm{min}$, $\alpha = \alpha_\textrm{min}$. At about $t=259$, $\alpha_\textrm{min}$ < $\mathcal{A}_\textrm{up}$ < $\alpha_\textrm{max}$, so the $\alpha$ command begins to follow $\mathcal{A}_\textrm{up}$, as predicted. Angle of attack saturates at $\alpha_\textrm{max}$ when $\mathcal{A}_\textrm{up}$ > $\alpha_\textrm{max}$ around $t = 268$. Similar to the linear case, we expect a $\sigma$ switch when $\mathcal{A}_\textrm{up} = \mathcal{A}_\textrm{down}$, which occurs near $t=271$. Bank angle switches from $\sigma_{\textrm{min}}$ to $\sigma_{\textrm{max}}$ at this time. Since the vehicle is flying lift-down ($u_1 \approx -1$), we next expect $\mathcal{A}_{\textrm{down}}$ to define $\alpha^*$. As expected, the angle of attack profile follows $\mathcal{A}_{\textrm{down}}$ when $\alpha_\textrm{min} < \mathcal{A}_{\textrm{down}} < \alpha_\textrm{max}$.

Despite no restriction on the control structure in GPOPS, the numerical results are consistent with the control theory presented previously in both the linear and quadratic case. This suggests that the theorems presented in Sec. \ref{sec: ocp solutions} are plausibly valid under our set of assumptions. The following section presents an aerocapture guidance algorithm that aims to exploit both aspects of the optimal control solution. 

\begin{figure}[t!]
	\centering 
	\begin{minipage}[b]{0.49\textwidth}
		\centering 
		\subcaptionbox{Linear aerodynamic model. Analytic switching functions (top), optimal $\sigma$ (center) and $\alpha$ (bottom) profiles from GPOPS control, state, and costate solution.  \label{fig: linear gpops}}{
			\includegraphics[width = \textwidth]{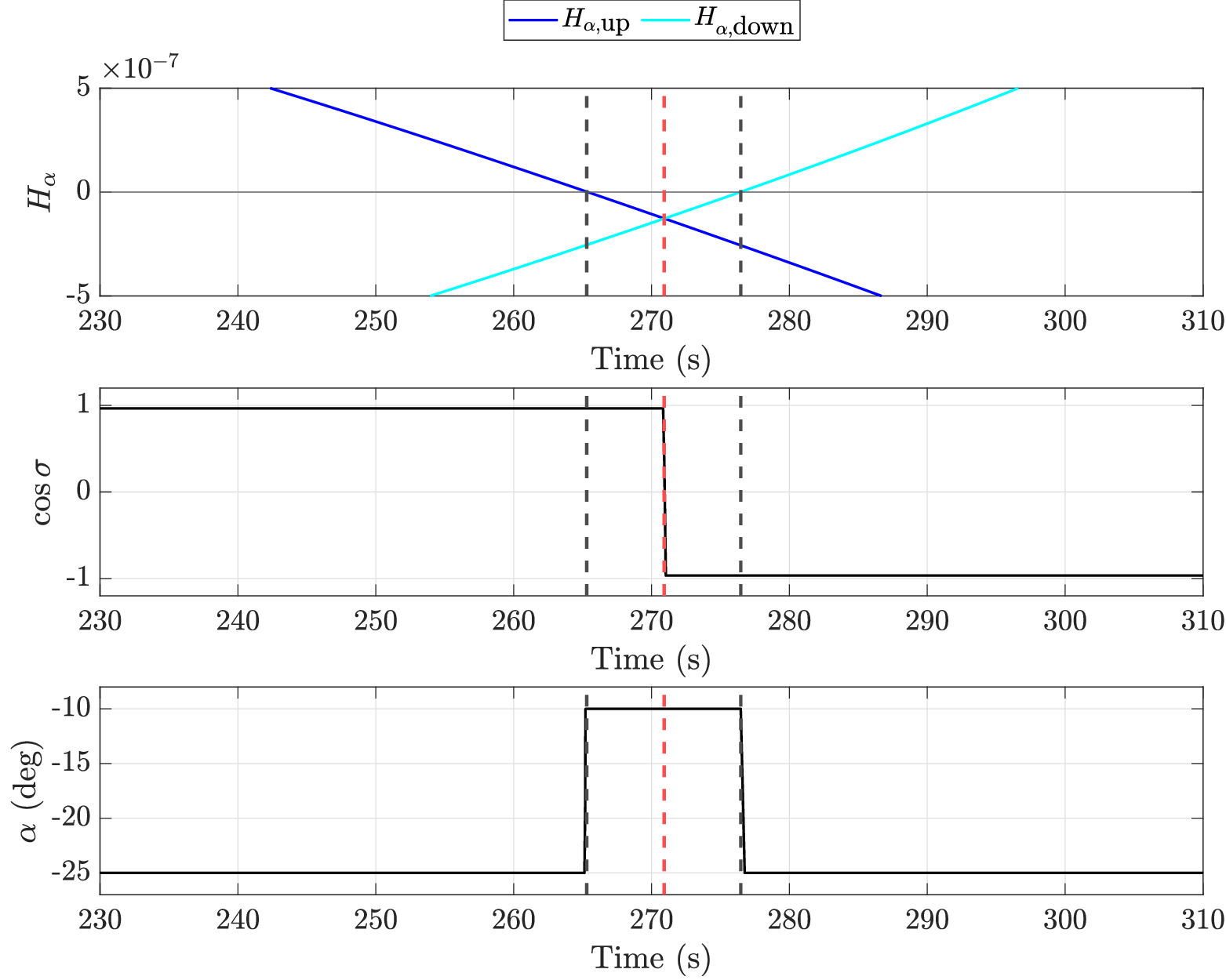}
		}
	\end{minipage}\hfill
	\begin{minipage}[b]{0.49\textwidth}
		\centering 
		\subcaptionbox{Quadratic aerodynamic model. Analytic control functions (top), optimal $\sigma$ (center) and $\alpha$ (bottom) profiles from GPOPS control, state, and costate solution.  \label{fig: quadratic gpops}}{
			\includegraphics[width = \textwidth]{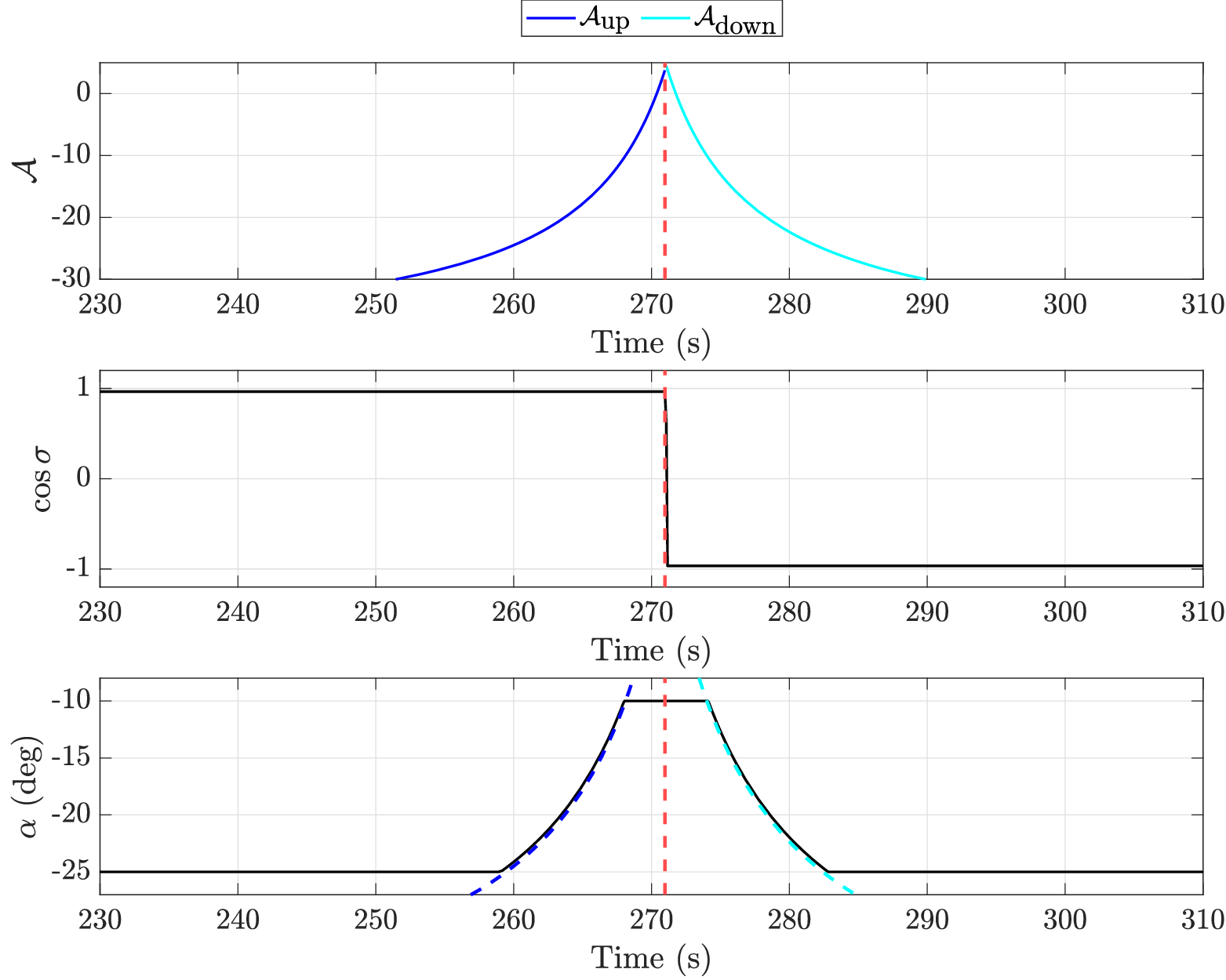}
		}
	\end{minipage}
	\caption{GPOPS results. Black vertical dashed lines correspond to $\alpha$ switching times, and red vertical dashed lines correspond to $\sigma$ switching times. With both models, the analytic and numerical solutions are consistent: in Fig. \ref{fig: linear gpops}, analytic switching times (zero crossings and curve intersection in top panel) align closely with the switching behavior in the GPOPS control profiles. In Fig \ref{fig: quadratic gpops}, the analytic control profiles (top panel) and bank angle switch (red dashed line) match the GPOPS control profiles.}
	\label{fig: GPOPS joint fig}
\end{figure}

%% file: Guidance_Algorithm_Design.tex
In this section we present ABAMGuid+, an aerocapture guidance algorithm based on the optimal control solutions presented in the previous section. To account for trajectory dispersions in a way that fully utilizes the benefits of adding angle of attack control to BAM, we introduce continuous alpha-sigma modulation (CASM), a novel subprocess inspired by the quadratic optimal control solution that improves the vehicle's ability to respond to off-nominal conditions. 

\subsection{Four-Phase Formulation}
ABAMGuid+ is a predictor-corrector algorithm that consists of four phases, each intended to exploit a specific part of the optimal control solution. A schematic of the control structure is given in Fig. \ref{fig: 4 phases}. 
The first three phases are intended to approximate the three-switch bang-bang control structure dictated by Theorem \ref{theorem: theorem 1} and Remark \ref{remark 3}, which was derived using the linear aerodynamic model in Problem \ref{problem: P2}. In these phases, the guidance predictor uses the bang-bang profile in Fig. \ref{fig: 4 phases}, but assumes continued saturation in phase 4. 
This structure is used due to the strength of the linear approximation, the simplicity of the resulting control profile, and the demonstrated robustness of this approach in \cite{sonandres2025aerocapture}.
By attempting to fly the optimal control solution,  $\Delta V$ minimization is incorporated indirectly, while enabling an efficient implementation that avoids having to solve the constrained, nonlinear optimal control problem each time the guidance system is invoked (which is intractable due to the nonconvexity of the problem).
Following the three switches, a new, unsaturated solution methodology is necessary, as maintaining saturated control for the remainder of the trajectory defines an open-loop policy that leaves guidance with no ability to respond to dispersions caused by flying through a stochastic atmosphere. This closed-loop terminal-phase control is done with CASM, which is intended to approximate the unsaturated control structure given in Theorem \ref{theorem: theorem 2}, which was derived using the quadratic aerodynamic model in Problem \ref{problem: P3}.
ABAMGuid+ is summarized in Algorithm \ref{alg: abamguid+}.
To account for modeling uncertainty, we deploy the filtering technique explained in \cite{lu2015optimal} to estimate the deviation of the sensed lift and drag accelerations from the expected accelerations based on a nominal model.

\begin{figure}[h!]
	\centering 
	\includegraphics[width = 0.6\textwidth]{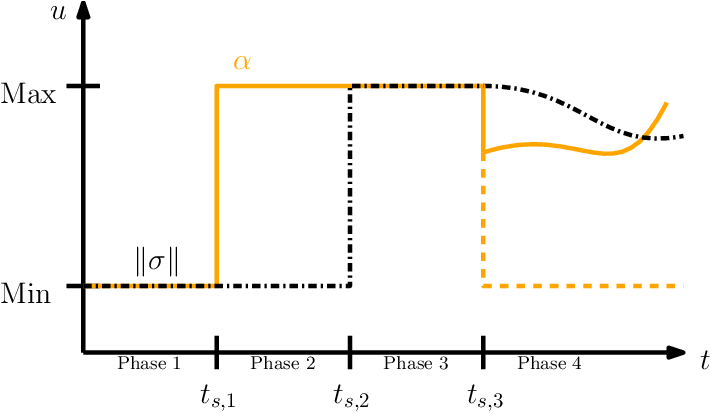}
	\caption{Illustration of the four-phase control structure used in ABAMGuid+. The nominal profile consists of three bang-bang switching phases, followed by an unsaturated terminal phase to account for trajectory dispersions. The dashed orange line represents the profile used in our previous method (ABAMGuid), which maintains a saturated angle of attack in the final phase. ABAMGuid+ has enhanced ability to respond to terminal phase dispersions by performing unsaturated control of both channels.}  
	\label{fig: 4 phases}
\end{figure}

\begin{algorithm}[t]
\caption{ABAMGuid+: Augmented Bank Angle Modulation aerocapture Guidance} \label{alg: abamguid+}
\begin{algorithmic}[1]
\Require Current vehicle state $\mathbf{x}$, current switching time solutions $t_{s,1}, t_{s,2}, t_{s,3}$, previous control $u_\textrm{prev}$
\Ensure Commands for angle of attack $\alpha$ and bank angle $\sigma$

\If{$t < t_{s,1}$}
    \State \textbf{Phase 1:}
    \State Set $\alpha \gets \alpha_{\min} $, $\sigma \gets  \sigma_{\min} $
    \State Solve for switching times $t_{s,1}, t_{s,2}, t_{s,3}$ using Nelder-Mead to minimize chosen objective function

\ElsIf{$t < t_{s,2}$}
    \State \textbf{Phase 2:}
    \State Set $\alpha \gets \alpha_{\max}$, $\sigma \gets \sigma_{\min} $
    \State Solve for switching times $t_{s,2}, t_{s,3}$ using Nelder-Mead to minimize chosen objective function

\ElsIf{$t < t_{s,3}$}
    \State \textbf{Phase 3:}
    \State Set $\alpha \gets  \alpha_{\max} $, $\sigma \gets \sigma_{\max} $
    \State Solve for final switching time $t_{s,3}$ using Newton-Raphson to minimize chosen objective function

\Else
    \State \textbf{Phase 4:} Invoke CASM subroutine to compute potentially unsaturated $(\alpha, \sigma)$
    \State $(\alpha, \sigma) \gets \textsc{CASM}(\mathbf{x}, u_\textrm{prev})$
\EndIf

\Return $(\alpha, \sigma)$
\end{algorithmic}
\end{algorithm}

\textbf{Phase 1} is defined over $t_0 \leq t < t_{s,1}$, which represents the interval from atmospheric entry interface to the first switching time, $t_{s,1}$, and is initially activated via a load trigger set to 0.1 g's. 
The load trigger ensures that there is sufficient control authority to begin active guidance. During this phase, the guidance predictor assumes the bang-bang bank angle and angle of attack control trajectories previously presented, parameterized by three switching time variables ($t_{s,1}$, $t_{s,2}$, $t_{s,3}$). At each guidance call, we use the Nelder-Mead method \cite{nelder1965simplex} to solve for the switching times which minimize an objective function based on the velocity targeting error. The objective function is
\begin{equation} \label{eq: NM objective}
f(t_{s,1}, t_{s,2}, t_{s,3}) = \frac{1}{2} \left( V_\text{exit, pred} - V_\text{exit}^* \right)^2,
\end{equation}
where $ V_\text{exit, pred}$ is the predicted velocity at atmospheric exit and $V_\text{exit}^*$ is the target exit velocity. A zeroth-order method is selected to avoid numerical issues possible when inverting the Jacobian of $f$, which is required for methods such as Gauss-Newton, used in \cite{lu2014entry}. The Nelder-Mead simplex $S$ is initialized by adding a constant value $\Delta t$ to each component of the initial guess ($t_{s,1}^0$, $t_{s,2}^0$, $t_{s,3}^0$), resulting in a simplex of $n=4$ points. Temporal consistency of the three switching times is enforced by penalizing violations of $t_{s,1} \leq t_{s,2} \leq t_{s,3}$ with large function values. Iteration is stopped once the stopping criteria is met
\begin{equation}
\text{std} \left( S_f \right) < \epsilon_{NM},
\end{equation}
where $S_f$ are the function values of each point in the simplex, std denotes standard deviation, and $\epsilon_{NM}$ is a preselected satisfactorily small value, set to 1 for velocity targeting. During phase 1, guidance commands $\sigma_\text{min}$ and $\alpha_\text{min}$, which reflects the initial control structure assumed in Remark \ref{remark 3}.

\textbf{Phase 2} spans the interval $t_{s,1} \leq t < t_{s,2}$. Similar to phase 1, we seek to find $t_{s,2}$ and $t_{s,3}$ which minimize Eq. \eqref{eq: NM objective} using the Nelder-Mead method. Because $t \geq t_{s,1}$, $t_{s,1}$ no longer has an impact on the objective function. Therefore phase 2 presents an optimization problem over two variables, reducing the computational complexity of the problem. The Nelder-Mead simplex size is now $n=3$, and is initialized identically to the previous phase. During phase 2, the guidance system commands $\sigma_\text{min}$ and $\alpha_\text{max}$. The $\alpha$ switch from $\alpha_\text{min}$ in phase 1 to $\alpha_\text{max}$ in phase 2 is dictated by Theorem \ref{theorem: theorem 1} following the zero-crossing of the lift-up switching curve $H_{\alpha, \textrm{up}}$, which is given in Definition \ref{def: alpha switching functions}.

\textbf{Phase 3} is defined from $t_{s,2} \leq t < t_{s,3}$. 
As done in phase 1 and 2, we seek to minimize the velocity targeting error function in Eq. \eqref{eq: NM objective}. 
Since $t$ is now greater than $t_{s,1}$ and $t_{s,2}$, the problem is now further reduced to a univariate optimization problem, as in phase 1 of FNPAG. We therefore minimize the objective function using the Newton-Raphson method and a secant scheme derivative approximation using the two previous iterations, as done in \cite{lu2014entry}. This update is
 
\begin{equation}
    t_{s,3}^{(k+1)} = t_{s,3}^{(k)} - \frac{z(t_{s,3}^{(k)})}{\left[ z(t_{s,3}^{(k)}) - z(t_{s,3}^{(k-1)}) \right]} \left( t_{s,3}^{(k)} - t_{s,3}^{(k-1)}\right),
\end{equation}
where $k$ and $k-1$ are the previous iterations used to compute the updated solution, and $z$ is defined as the velocity targeting error $V_\text{exit, pred} - V_\text{exit}^* $. Iteration is stopped when the stopping criteria is met
\begin{equation}
    \| z(t_{s,3}^{(k+1)}) \frac{\partial z(t_{s,3}^{(k+1)})}{\partial t_{s,3}}  \| \leq \epsilon_{NR},
\end{equation}
where $\epsilon_{NR}$ is a preselected small value. During phase 3, guidance commands $\sigma_{\textrm{max}}$ and $\alpha_{\textrm{min}}$. The $\sigma$ switch from $\sigma_\text{min}$ in phase 2 to $\sigma_\text{max}$ in phase 3 approximates the intersection of the lift-up and lift-down switching curves from Definition \ref{def: alpha switching functions}, which was established to correspond to a bank angle switch in Remark \ref{remark: bank switch}.

\textbf{Phase 4} is defined over $t_{s,3} \leq t < t_f$, where $t_f$ is the time at atmospheric exit. 
Here, we seek a constant magnitude combination of $\alpha$ and $\sigma$ that satisfies the velocity targeting criteria using the CASM subprocess. Finding this combination defines a two-variable control problem. CASM enables the use of Brent's method via reformulation of the two-variable problem into a one-dimensional line search, effectively reducing the complexity of the two-variable problem. This phase is primarily designed to respond to trajectory dispersions by utilizing both control channels in a way that considers the coupling between the commands, thus maximizing effectiveness. The algorithm is detailed in the following section. 
 
\subsection{Continuous Alpha-Sigma Modulation} \label{sec: casm}
The CASM algorithm is summarized in Algorithm \ref{alg: casm}. CASM is inspired by the quadratic optimal control solution in Theorem \ref{theorem: theorem 2}, and more specifically the curve $\mathcal{A}_\textrm{down}$ in Definition \ref{def: optimal alpha curves}, which suggests that there are some lift-down flight regimes where it is beneficial to perform unsaturated $\alpha$ control. The optimal control structure is not directly parameterized, as in phases 1-3, due to the added complexity associated with modeling $\mathcal{A}_\textrm{down}$. Instead, CASM is intended to approximate the unsaturated control $\mathcal{A}_\textrm{down}$ in a way that prioritizes simplicity and efficiency, while maintaining effectiveness. Although this approach may deviate slightly from the optimal trajectory, the ability to directly modulate the lift and drag vectors enables an improved response to dispersions, ultimately yielding significant performance benefits over previous methods. This is demonstrated through simulation in Sec. \ref{sec: guidance performance} and Sec. \ref{sec: mc results}. Intuitively, CASM is beneficial because maintaining a fixed $\alpha$ in the unsaturated control regime reverts the system to traditional BAM (where the magnitudes of the lift and drag vectors are no longer controllable), which diminishes the vehicle's ability to achieve the target orbit in the presence of dispersions. CASM can be broken into two main steps: \textit{initialization}, 
in which we sample the control search space in a way that fully characterizes the set of achievable objective function outcomes, and \textit{solve}, where we find a command that satisfies the targeting criterion (or minimizes error in cases where it is not possible).  
 
 \begin{algorithm}[t!] 
\caption{CASM: Continuous Alpha-Sigma Modulation (Phase 4)} \label{alg: casm}
\begin{algorithmic}[1]
\Require Current vehicle state $\mathbf{x}$, previous control $u_\textrm{prev}$
\Ensure Control pair $(\alpha, \sigma)$

\State Initialize control space using corner points $(\sigma, \alpha) \in \{\sigma_{\min}, \sigma_{\max}\} \times \{\alpha_{\min}, \alpha_{\max}\}$, previous control $(\sigma_{\text{prev}}, \alpha_{\text{prev}})$, and optionally $(\sigma_{\max/\min}, \alpha_{\text{prev}})$
\State Evaluate $f(\sigma, \alpha)$ at each initialization point
\State Select two control points $(\sigma_a, \alpha_a)$ and $(\sigma_b, \alpha_b)$ such that $f_a < 0$ and $f_b > 0$
\If{no valid bracket is found}

    \Return $\underset{u}{\arg \min} \| f(u) \|$ from all evaluated points
\EndIf

\State Define objective: $f(\kappa) = V_{\text{exit,pred}}(\sigma(\kappa), \alpha(\kappa)) - V^*_{\text{exit}}$, where $\sigma(\kappa) = \sigma_a + \kappa(\sigma_b - \sigma_a)$,  $\alpha(\kappa) = \alpha_a + \kappa(\alpha_b - \alpha_a)$, and $\kappa \in [0, 1]$
\State Use Brent's method to find $\kappa^*$ such that $f(\kappa^*) = 0$

\Return $(\alpha(\kappa^*), \sigma(\kappa^*))$
\end{algorithmic}
\end{algorithm}
\textit{Initialization.} We begin by defining the control search space corners using all pairs of $\sigma$ $\in$ $\{ \sigma_{\text{min}}$, $\sigma_{\text{max}}\}$ and $\alpha$ $\in$ $\{ \alpha_{\text{min}}$, $\alpha_{\text{max}} \}$, defining the control point set $\mathcal{I}$. For each control point $u_p = (\sigma_p, \alpha_p) \in \mathcal{I}$, the trajectory is propagated forward and the resulting objective function value $f(u_p) = V_\text{exit, pred} - V^*_\text{exit}$ associated with the point is stored. Assuming that $f(u_p)$ is monotonic and continuous in $\alpha$ and $\sigma$, which is valid given the exit velocity reformulation and restriction of $\alpha$ to negative values, this entirely defines the set of attainable objective function outcomes. Importantly, this tells us if a solution that satisfies the velocity targeting criterion exists ($\exists u_x, u_y \in \mathcal{I}, \text{ s.t. } f(u_x) \cdot f(u_y) < 0$). The previously commanded guidance solution, ($\sigma_{\text{prev}}$, $\alpha_{\text{prev}}$), is stored in $\mathcal{I}$ as a fifth initialization point, along with its recomputed targeting function value. Optionally, we initialize a sixth solution in $\mathcal{I}$ at ($\sigma_{\text{max/min}}$, $\alpha_{\text{prev}}$) depending on the objective value of the previous solution. This sixth evaluation is not strictly necessary for the remainder of the algorithm, but it provides an interpolation option that prioritizes BAM.

\textit{Solve.} We ultimately desire a solution to the velocity targeting function, where we assume a constant $\alpha$ and $\sigma$. So, we wish to find the root of 
\begin{equation} \label{eq: vtarg phase4}
f(\alpha, \sigma) = V_\text{exit, pred} - V_\text{exit}^*.
\end{equation}
We proceed by finding two zero-bracketing control points in $\mathcal{I}$, $u_a = \left(\sigma_a, \alpha_a \right)$ and $u_b = \left(\sigma_b, \alpha_b \right)$, such that $f(u_a) < 0$ and $f(u_b) > 0$. To discourage large control jumps between guidance calls, the previous solution $u_{i-1}$, where $i$ represents the $i$-th guidance call, is used as a bracketing point. That is, $u_a = u_{i-1}$ if $f_{i-1} < 0$, and $u_b = u_{i-1}$ if $f_{i-1} > 0$. The remaining point is selected among the initialization points such that sign($f$) = -sign($f_{i-1}$), and $\| f \|$ is smallest, providing the tightest bracketing. We next define a linear parameterization of $\sigma$ and $\alpha$ as functions of $\kappa$, where $\kappa$ varies from 0 to 1. This relationship is given as
\begin{equation} \label{eq: sigma parameterization}
\sigma(\kappa) = \sigma_a + \kappa (\sigma_b - \sigma_a),
\end{equation}
\begin{equation} \label{eq: alpha parameterization}
\alpha(\kappa) = \alpha_a + \kappa ( \alpha_b - \alpha_a).
\end{equation}
Note that $\kappa=0$ corresponds to $u_a = (\sigma_a, \alpha_a)$, and $\kappa=1$ corresponds to $u_b = (\sigma_b, \alpha_b)$. This linear parameterization enables a single dimensional line search over the placeholder variable, $\kappa$. We can rewrite the objective function as a function of $\kappa$ as 
\begin{equation} \label{eq: parameterized objective}
f(\kappa) = f(\sigma_a + \kappa ( \sigma_b - \sigma_a ), \alpha_a + \kappa ( \alpha_b - \alpha_a) = V_\text{exit, pred} - V^*_\text{exit}.
\end{equation}
Brent's method \cite{brent2013algorithms} is used to find the root of Eq. \eqref{eq: parameterized objective}. Brent's method combines the speed of inverse quadratic interpolation and the secant method with the guaranteed convergence of the bisection method while not relying on computing derivatives. Once $\kappa$ is found, the guidance commands can be recovered through Eq. \eqref{eq: sigma parameterization} and Eq. \eqref{eq: alpha parameterization}, and guidance commands $\sigma(\kappa)$ and $\alpha(\kappa)$. Note that if no bracketing exists, the solve step is skipped and guidance selects a solution from the initialization points that minimizes targeting error, given by
\begin{equation}
u = \underset{u \in \mathcal{I}}{\arg \min} \| f(u) \|.
\end{equation}
As with FNPAG, it is possible to bypass the switching time selection logic entirely, instead solely using CASM for the full duration of the flight. 
However, this approach does not fly the bang-bang control structure derived from solving Problem \ref{problem: P2}, and thus may lead to worse $\Delta V$ performance.

%% file: Guidance_Performance.tex
This section presents individual closed loop Uranus Orbiter and Probe (UOP) aerocapture simulation results. We compare guidance outputs and performance metrics of the proposed algorithm, ABAMGuid+, to FNPAG and ABAMGuid. FNPAG was implemented following the derivation in \cite{lu2015optimal}. ABAMGuid was implemented as in our previous work \cite{sonandres2025aerocapture}, but is adapted to use the exit velocity targeting objective. Two cases are presented. The nominal entry case is discussed in Sec. \ref{subsec: nominal entry}, where each algorithm is expected to perform well.
An off-nominal steep case is presented in Sec. \ref{fig: steep trajectory}, which highlights ABAMGuid+'s ability to successfully capture the vehicle in a case where FNPAG and ABAMGuid do not. 
\subsection{Simulation Methodology}
Throughout this work, aerocapture trajectories are simulated using a high-fidelity 3-DoF simulation developed in MATLAB/ Simulink. The full nonlinear dynamics are simulated at 100 Hz while the guidance system is called at 2 Hz. The vehicle's mass properties follow the Maximum Possible Value (MPV) configuration in \cite{gomez2024design}, and the aerodynamic model is similar to that of the Mars Science Laboratory \cite{shellabarger2025aerodynamics}. Uranus-GRAM \cite{justh2021uranus} is used to model the planet's atmosphere. 
Gravity is modeled assuming a non-spherical body  using spherical harmonic terms up to the J2 coefficient. The on-board atmospheric density model is created by fitting a piece-wise polynomial to the natural log of the nominal GRAM profile. Simulations are classified as successful if the post-aerocapture period is between 10 days and 2.5 years.
Table \ref{table: MC dispersions} summarizes the vehicle properties, initial state conditions, and orbital parameter targets used in simulation. The dispersion ranges represent the ranges of values used for Monte Carlo testing. 
Angle of attack and bank angle are initialized to -17 degrees and -165 degrees, respectively, for each algorithm.
 
\begin{table}[htb]
\small
\caption{Vehicle and mission parameters used in UOP simulation.}
\vspace{-6mm}
\label{table: MC dispersions}
\begin{center}
\resizebox{\textwidth}{!}{
\begin{tabular}{c c | c c}
\toprule
\toprule
\textbf{Mission Parameter} & \textbf{Value} & \textbf{Vehicle Parameter} & \textbf{Value} \\
\midrule
EI Altitude, $h_0$ (km) & 1,000 &
Mass, $m_0$ (kg) & 4,063 \\
Inertial Velocity, $V_0$ (km/s) & 23.78 &
Nominal $L/D$ ratio & $0.25 \pm 0.005$ 3-$\sigma$ \\
Inertial Entry Flight Path Angle, $\gamma_0$ (deg) & $-10.79 \pm 0.189$ 3-$\sigma$ (Baseline) &
Bank Angle Limits (deg) & $15 \leq \|\sigma\| \leq 165$ \\
& $-10.79 \pm 0.622$ 3-$\sigma$ (Conservative) &
Angle of Attack Limits (deg) & $-25 \leq \alpha \leq -10$ \\
Entry Longitude, $\theta$ (deg) & 262.12 &
Angle of Attack Rate Limit (deg/sec) & 5 \\
Entry Latitude, $\phi$ (deg) & -16.02 &
Bank Angle Rate Limit (deg/sec) & 15 \\
Inertial Entry Azimuth Angle, $\psi_0$ (deg) & 117.45 &
 & \\
Target Apoapsis Altitude (km) & 2,000,000 &
 & \\
Target Periapsis Altitude (km) & 4,000 &
 & \\
\bottomrule
\bottomrule
\end{tabular}
}
\end{center}
\end{table}

\subsection{Nominal Entry} \label{subsec: nominal entry}
The first test case is the nominal entry scenario, in which there are no entry or atmospheric dispersions. This serves as an effective validation benchmark, since the unstressful conditions should result in excellent performance. Figure \ref{fig: nominal trajectory} shows results for the nominal entry condition ($\gamma$ = -10.79 degrees) for ABAMGuid+, ABAMGuid, and FNPAG. The upper panels display relevant trajectory information. The top-left plot reveals a typical U-shaped curve that indicates most velocity reduction occurs at lower altitude. The three trajectories are nearly identical, and all achieve the desired exit velocity target. The top-right panel shows the evolution of orbital eccentricity from a hyperbolic approach trajectory ($\epsilon$ > 1) to an elliptical orbit ($\epsilon$ < 1) via the atmospheric pass. Again, each algorithm results in a similar profile ending at the target eccentricity. The bottom panels highlight the difference in control profiles flown by each method. Each algorithm commands a very brief lift-up phase, but transitions into the next phase before the command can be tracked. This is presumably due to guidance tending to skip the lift-up phase as a result of the highly eccentric target orbit and shallow nominal entry condition. FNPAG proceeds to perform bank angle modulation while maintaining a constant angle of attack. ABAMGuid+ and ABAMGuid fly phase 3, commanding lift-down and minimum angle of attack, before entering phase 4, where the two algorithms diverge in behavior. By design, ABAMGuid commands a constant $\alpha$ at maximum magnitude, while ABAMGuid+ solves for a constant $\alpha$-$\sigma$ pair using CASM. The $\Delta V$ required to correct the orbit following the atmospheric pass is 18.9 m/s for ABAMGuid+, 19.4 m/s for ABAMGuid, and 19.4 m/s for FNPAG. This difference is insignificant and each algorithm achieves a captured orbit with minimal $\Delta V$ required for orbital correction. 

\begin{figure}[t!]
\centering
\includegraphics[width=0.8\textwidth]{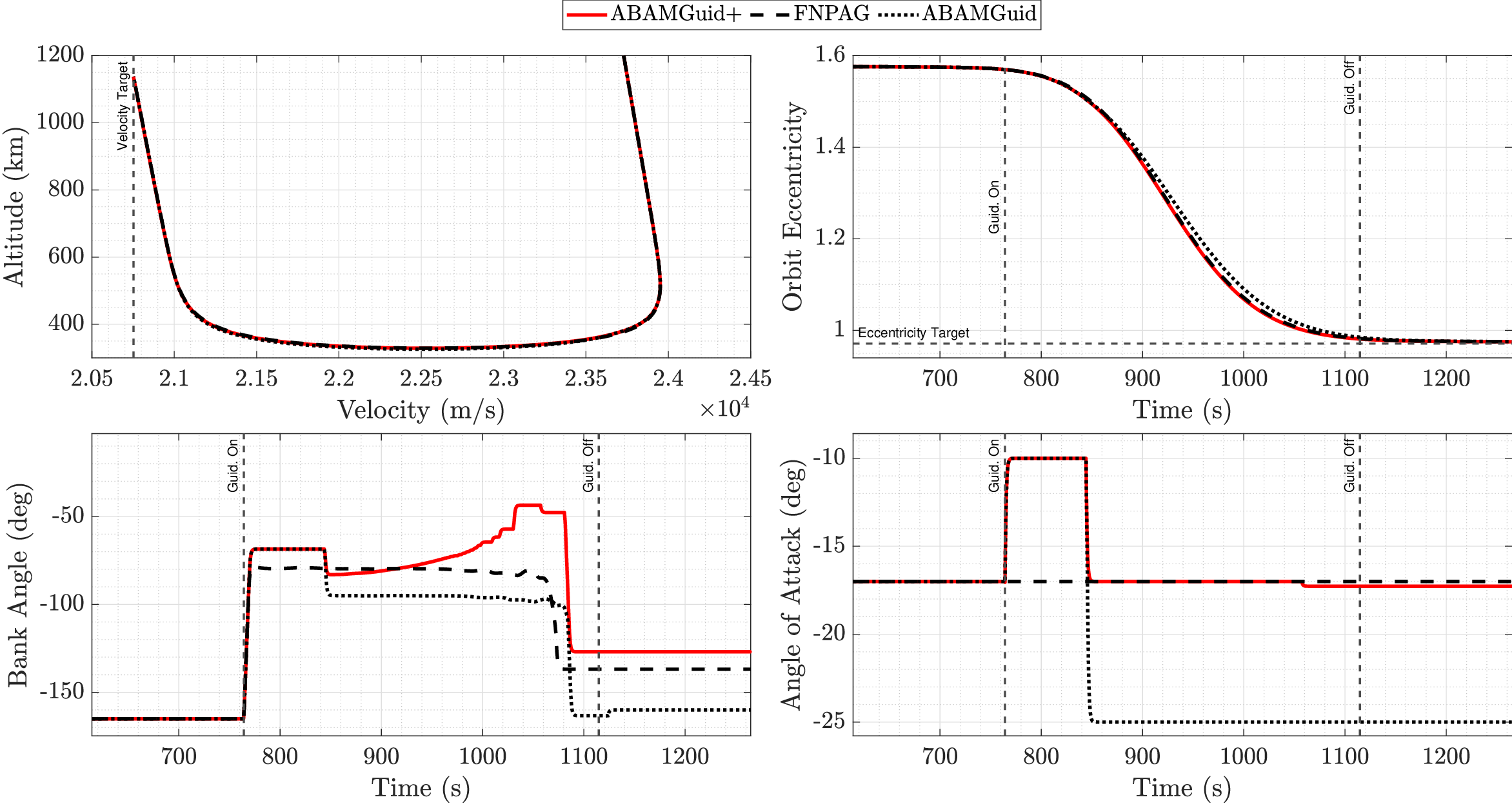} 
\caption{Nominal trajectory ($\gamma$ = -10.79 degrees) results for each algorithm. The altitude vs. velocity profile (top-left) and evolution of orbital eccentricity (top-right) illustrate that each trajectory converges on the target exit velocity and orbital eccentricity. The $\sigma$ and $\alpha$ profiles over time are shown in the bottom-left and bottom-right panels, respectively. } \label{fig: nominal trajectory}
\end{figure}

\subsection{Off-Nominal Entry} \label{subsec: steep entry}
The second test case is an off-nominal, extremely steep entry condition ($\gamma$ = -11.12 degrees). A steep entry flight path angle is challenging due to the large amounts of drag encountered when flying through lower than anticipated altitude ranges. Figure \ref{fig: steep trajectory} illustrates an example trajectory where FNPAG and ABAMGuid fail to capture into the desired orbit, but ABAMGuid+ succeeds. The top-left altitude versus velocity plot shows that FNPAG severely undershoots the exit velocity target. ABAMGuid gets significantly closer to the target, but is still unable to exit the atmosphere with sufficient velocity. ABAMGuid+ is able to maintain slightly more exit velocity, leading to a successful capture. The top-right orbital eccentricity plot confirms this; too much velocity is depleted using FNPAG and ABAMGuid, and the resulting orbital eccentricity is lower than targeted. ABAMGuid+ is again closest to the target. The control trajectories in the bottom panels indicate that FNPAG immediately fully saturates the bank angle channel to fly lift-up, and because there is no ability to perform angle of attack control, FNPAG is performing optimally for a BAM-based guidance algorithm. ABAMGuid and ABAMGuid+ are able to modulate $\alpha$, and therefore add control authority by saturating $\alpha$. Just after 800 seconds, which is a high-density flight regime near periapsis, both algorithms unsaturate their outputs. This is likely due to an atmospheric density perturbation that causes the guidance system to overestimate the remaining control authority. ABAMGuid maintains a maximum magnitude $\alpha$ and commands a $\sigma$ response that it is unable to recover from. ABAMGuid+ similarly responds to the environmental dispersion, but in performing simultaneous $\alpha$-$\sigma$ control via CASM, ABAMGuid+ has increased ability to respond and recover from dispersions, leading to a successful capture. FNPAG results in an orbital period of 1.22 days and a $\Delta V$ requirement of 1273.3 m/s, and ABAMGuid results in an orbital period of 8.36 days with $\Delta V$ = 248.4 m/s. Both are categorized as failed simulations. ABAMGuid+ results in an orbital period of 10.69 days and $\Delta V$ = 189.2 m/s, which is classified as successful. 
 
\begin{figure}[t!]
\centering
\includegraphics[width=0.8\textwidth]{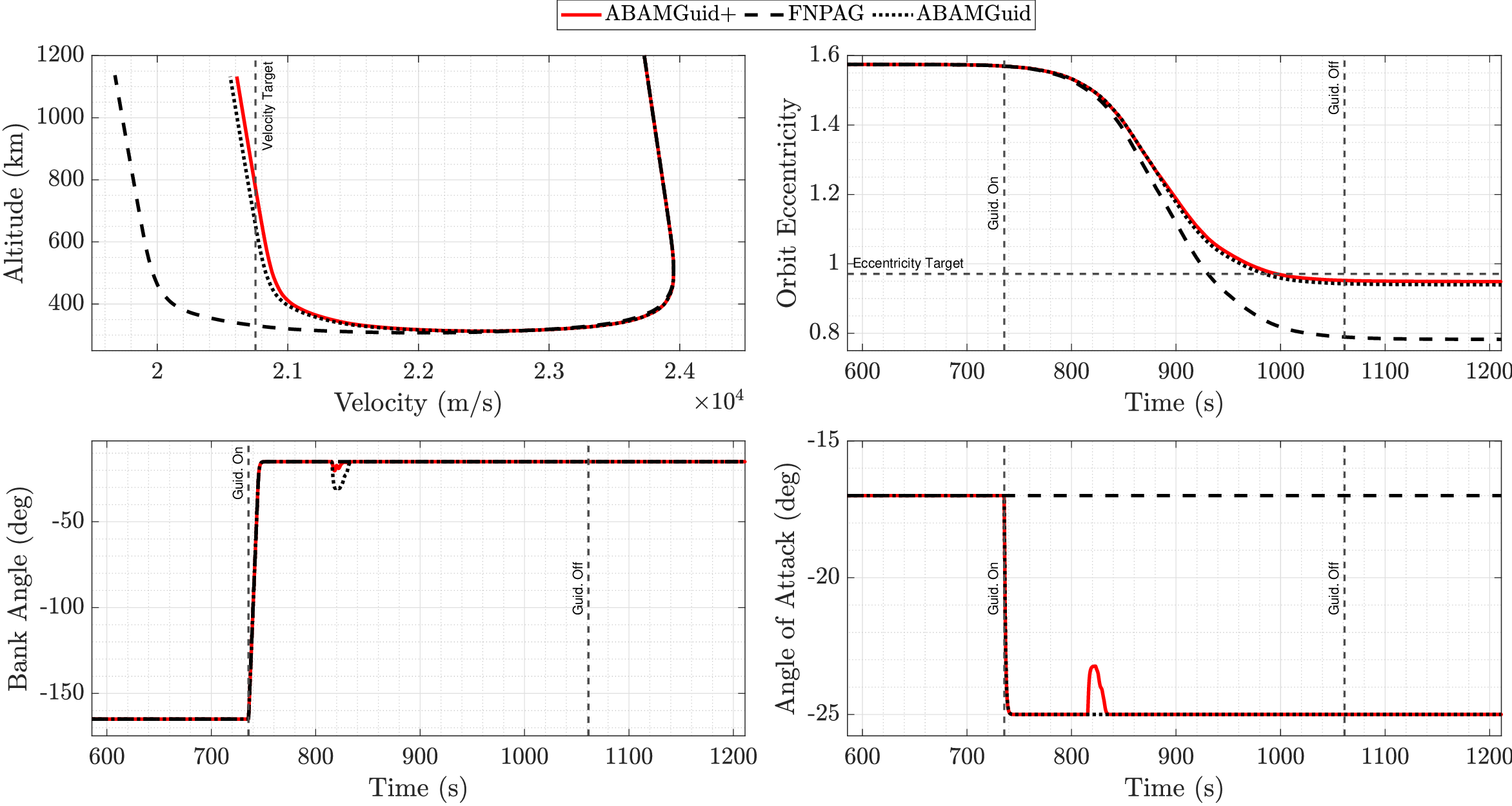} 
\caption{Steep trajectory ($\gamma$ = -11.06 degrees) results for each algorithm. The altitude vs. velocity profile (top-left) and evolution of orbital eccentricity (top-right) illustrate that ABAMGuid+ is closest to the target orbit, while ABAMGuid and FNPAG undershoot the target. The guidance commanded $\sigma$ and $\alpha$ over time are shown in the bottom-left and bottom-right panels, respectively.} \label{fig: steep trajectory}
\end{figure}

%% file: Sim_Results.tex
This section presents Monte Carlo results of a closed-loop Uranus Orbiter and Probe simulation scenario using ABAMGuid+, FNPAG, and our previous algorithm, ABAMGuid. 
\subsection{Monte Carlo Simulation Setup} \label{subsec: MC setup}
Monte Carlo simulations were conducted using ABAMGuid+, FNPAG, and ABAMGuid with 8,000 dispersed entry states, where the arrival states have been sourced from \cite{mages2025mission}. Two entry state sets are used, covering a range of EFPA uncertainty. The \textit{baseline} entry state set has an EFPA uncertainty range representative of a range tolerable by the baseline mission design. In the baseline set, we expect all algorithms to achieve nearly 100\% capture success due to the moderate dispersion levels. The EFPA dispersion range is greatly increased in the worst case \textit{conservative} entry state set. This entry set is selected to stress test the algorithms by subjecting the vehicle to extreme entry conditions. In the conservative set, many trajectories are expected to fall outside the theoretical entry corridor, and it follows that we expect decreased capture success rates. For a more detailed description of the entry state sets, the reader is referred to Ref. \cite{mages2025mission}. For each simulation, the true density profile is varied by sampling the Uranus-GRAM atmosphere model. In all test cases, the vehicle's aerodynamic coefficients are dispersed according to the methodology explained in \cite{schoenenberger2014assessment}.

Three testing cases are presented. Atmospheric dispersions and sensor noise are turned off in \textit{Test Case 1}. This case is intended to identify the theoretical limits of performance by granting the algorithms near-perfect information. Atmospheric dispersions are added in with perfect sensors and state knowledge in \textit{Test Case 2}, which introduces the dominant form of aerocapture uncertainty. Atmospheric dispersions and sensor noise are enabled in \textit{Test Case 3}. This case is intended to ensure the method is robust to imperfect state knowledge. The Monte Carlo simulation parameters and dispersions are summarized in Table \ref{table: MC dispersions}. For each test case, six total Monte Carlo sets are presented: FNPAG baseline set, FNPAG conservative set, ABAMGuid baseline set, and ABAMGuid conservative set, ABAMGuid+ baseline set, and ABAMGuid+ conservative set. For each set, $\Delta V$ statistics are computed using passing simulations. Capture success rates and resulting entry corridor width are also included for comparison. For \textit{Test Case 3}, the vehicle's state is estimated using the Extended Kalman Filter. The navigation filter is initialized with imperfect state knowledge of the initial state. The simulation models a high-quality IMU, described in \cite{chadalavada2025onboard}, incorporating the effects of a constant bias, scale factor uncertainty, axis non-orthogonality, and random walk.

\subsection{Results}
Table \ref{tab: aerocapture performance table} summarizes Monte Carlo performance across the three test cases and two entry sets, comparing ABAMGuid+, ABAMGuid, and FNPAG. In \textit{Test Case 1} baseline entry set, where environmental dispersions are disabled and entry dispersions are small, all three algorithms perform similarly. This is expected given the lack of perturbations and near-perfect knowledge available to each guidance system. In the conservative set, ABAMGuid and ABAMGuid+ also show near-identical performance, indicating that in low-uncertainty conditions the additional CASM logic does not yield significant gains. However, under high-uncertainty conditions (\textit{Test Case 2} and \textit{Test Case 3}), ABAMGuid+ consistently outperforms both ABAMGuid and FNPAG across all metrics. In the baseline sets, ABAMGuid+ achieves reductions in $\Delta V$ 3-$\sigma$ and 99-th \%-ile of about 65\% and 45\%, respectively, compared with ABAMGuid. In the conservative sets, we observe a 23\% improvement in mean $\Delta V$ over ABAMGuid. Across all cases, ABAMGuid+ maintains comparable or improved capture success rates and entry corridor widths, indicating no loss of robustness in any scenario. In the high-dispersion test cases, ABAMGuid+ results in about a 5\% reduction in cases that fail to be captured compared to ABAMGuid and a 50\% reduction compared to FNPAG. ABAMGuid+ reduces mean $\Delta V$ by up to 42\% in comparison to FNPAG in the baseline set (\textit{Test Case 3}), and up to 37\% in the conservative set. In the same test case and entry sets, ABAMGuid+ improves 99-th \%-ile $\Delta V$ by 56\% and 24\% over FNPAG. 

\begin{table}[t!]
\scriptsize
\centering
\caption{Monte carlo performance metrics for each test case and entry set. Bold indicates strongest performance.}
\label{tab: aerocapture performance table}
\begin{tabular}{c c c c c c c c}
\toprule
\toprule
\textbf{Test Case} & \textbf{Entry Set} & \textbf{Algorithm} & \(\Delta V\) Mean (m/s) & \(\Delta V\) 3-\(\sigma\) (m/s) & \(\Delta V\) 99th \%-ile (m/s) & Failures (Pass \%) & Entry Corridor (width) (deg) \\
\midrule
\multirow{6}{*}{1} 
 & \multirow{3}{*}{Baseline} 
 & FNPAG       & 19.4 & 0.5 & 19.8 & 0 (100) & --\\
 &                 & ABAMGuid  & 19.4 & 0.5 & 19.8  & 0 (100) & --\\
 &                 & ABAMGuid+  & \textbf{18.9} & \textbf{0.4} & \textbf{19.3} & 0 (100) &  --\\
\cmidrule(lr){2-8}
 & \multirow{3}{*}{Conservative} 
 & FNPAG       & 24.9 & 74.1 & 167.0 & 2032 (74.6) & [-11.002, -10.587] (0.414)\\
 &                 & ABAMGuid  & 20.8 & \textbf{36.9} & 87.4  & \textbf{1002 (87.48)} & [-11.153, -10.565] (0.588)\\
 &                 & ABAMGuid+  & 20.8 & 37.5 & 87.4 & 1004 (87.45) &  \textbf{[-11.158, -10.568] (0.590)}\\
\midrule
\multirow{6}{*}{2} 
 & \multirow{3}{*}{Baseline} 
 & FNPAG       & 37.4 & 52.1 & 98.7 & 2 (99.99) & --\\
 &                 & ABAMGuid  & 33.4 & 41.1 & 81.2 & 0 (100) & --\\
 &                 & ABAMGuid+  & \textbf{21.8} & \textbf{14.4} & \textbf{44.9} &0 (100) &-- \\
\cmidrule(lr){2-8}
 & \multirow{3}{*}{Conservative} 
 & FNPAG       & 50.1 & 99.9 & 179.7 & 2358 (70.53) & [-11.013, -10.589] (0.425)\\
 &                 & ABAMGuid  & 41.2 & \textbf{71.4} & \textbf{125.2}  & 1245 (84.44)& [-11.156, -10.574] (0.582) \\
 &                 & ABAMGuid+  & \textbf{32.0} & 72.9 & 136.4 & \textbf{1177} (85.29)  & \textbf{[-11.155, -10.565] (0.590)} \\
\midrule
\multirow{6}{*}{3} 
 & \multirow{3}{*}{Baseline} 
 & FNPAG       & 38.0 & 53.5 & 100.2 & 1 (99.99) & --\\
 &                 & ABAMGuid  & 33.9 & 42.1 & 82.8  & 0 (100) & --\\
 &                 & ABAMGuid+  & \textbf{22.2} & \textbf{14.8} & \textbf{44.3} & 0 (100) & --\\
\cmidrule(lr){2-8}
 & \multirow{3}{*}{Conservative} 
 & FNPAG       & 52.2 & 101.3 & 181.1 & 2355 (70.56) & [-11.008, -10.589] (0.419)
\\
 &                 & ABAMGuid  & 43.0 & 74.3 & \textbf{128.3}  & 1257 (84.29) & [-11.161, -10.575] (0.587)
\\
 &                 & ABAMGuid+  & \textbf{33.1} & \textbf{74.0} & 138.0 & \textbf{1189 (85.14)} & [-11.153, -10.566] (0.587)
\\
\bottomrule
\bottomrule
\end{tabular}
\end{table}

Figure \ref{fig: baseline mc figs} compares the $\Delta V$ distribution for the baseline entry set under \textit{Test Case 2} for each algorithm. The left panel shows $\Delta V$ versus EFPA, illustrating that ABAMGuid+ (red) consistently results in lower and more tightly clustered $\Delta V$ values than ABAMGuid (blue) and FNPAG (black). FNPAG and ABAMGuid show a much broader spread across the EFPA range, indicating larger variance in $\Delta V$ outcomes. The CDF in the right panel reinforces this result, as the ABAMGuid+ curve exhibits the sharpest rise and reaches a cumulative probability over 0.95 by approximately 30 m/s, indicating that most simulations achieve low-$\Delta V$ values. In contrast, the black and blue curves rise more gradually, implying that many more trajectories result in high-$\Delta V$ values. These results suggest that ABAMGuid+ is able to achieve stronger exit velocity targeting in the presence of environmental uncertainty due to the subprocess CASM improving the vehicle's ability to respond to dispersions, reducing propellant usage.

\begin{figure}[tb!]
	\centering 
	\includegraphics[width = 0.7\textwidth]{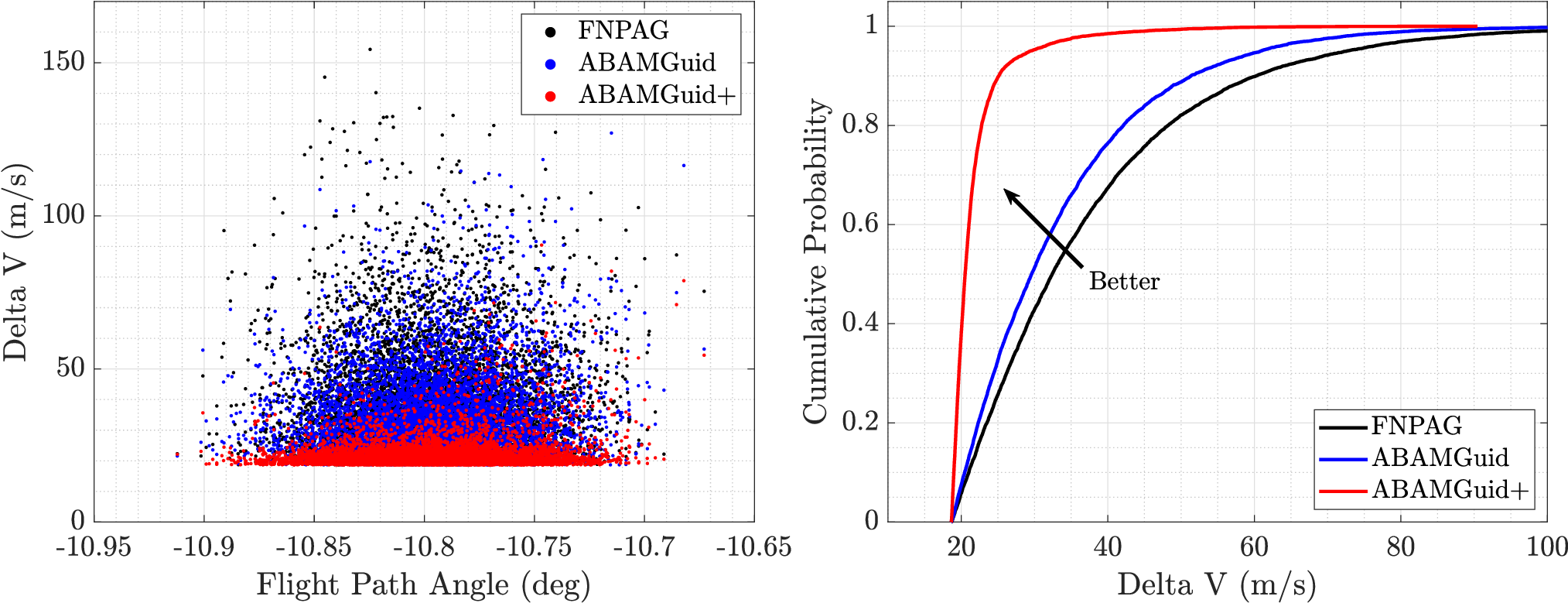}
	\caption{Monte Carlo $\Delta V$ results for the baseline entry state set (Test Case 2). A scatter plot of $\Delta V$ vs EFPA is shown in the left panel, and the $\Delta V$ CDF is shown in the right panel. ABAMGuid+ (red) results in a tighter distribution around low-$\Delta V$ values in comparison to FNPAG (black) and ABAMGuid (blue). Simulations which fail to be captured are omitted. } 
	\label{fig: baseline mc figs}
\end{figure}

Figure \ref{fig: conservative mc figs} presents $\Delta V$ results for the conservative entry set in \textit{Test Case 2}, which increases the range of EFPA dispersions. In the left panel, ABAMGuid+ again results in the lowest and most compact $\Delta V$ outcomes in a range around the nominal EFPA ($-10.9^\circ$ to $-10.7^\circ$). In the steep and shallow edges of the corridor, ABAMGuid+ and ABAMGuid perform similarly. Vertical dashed lines contain 99.7\% of successful simulations, defining the entry corridor width (ECW) for each method. ABAMGuid+ and ABAMGuid produce similar ECWs, with ABAMGuid+ having a slight advantage in shallow entry scenarios. ABAMGuid+ and ABAMGuid result in the same steep corridor boundary in the left of the plot, and as such the red and blue lines are on top of each other. Both methods are significantly wider than FNPAG (black dashed lines), highlighting their improved ability to handle extreme entry scenarios. The CDF in the right panel further confirms the $\Delta V$ advantage ABAMGuid+ achieves compared to ABAMGuid and FNPAG: the red line rises faster than the blue and black lines, indicating an increase in low-$\Delta V$ simulation results. 

\begin{figure}[tb!]
	\centering 
	\includegraphics[width = 0.7\textwidth]{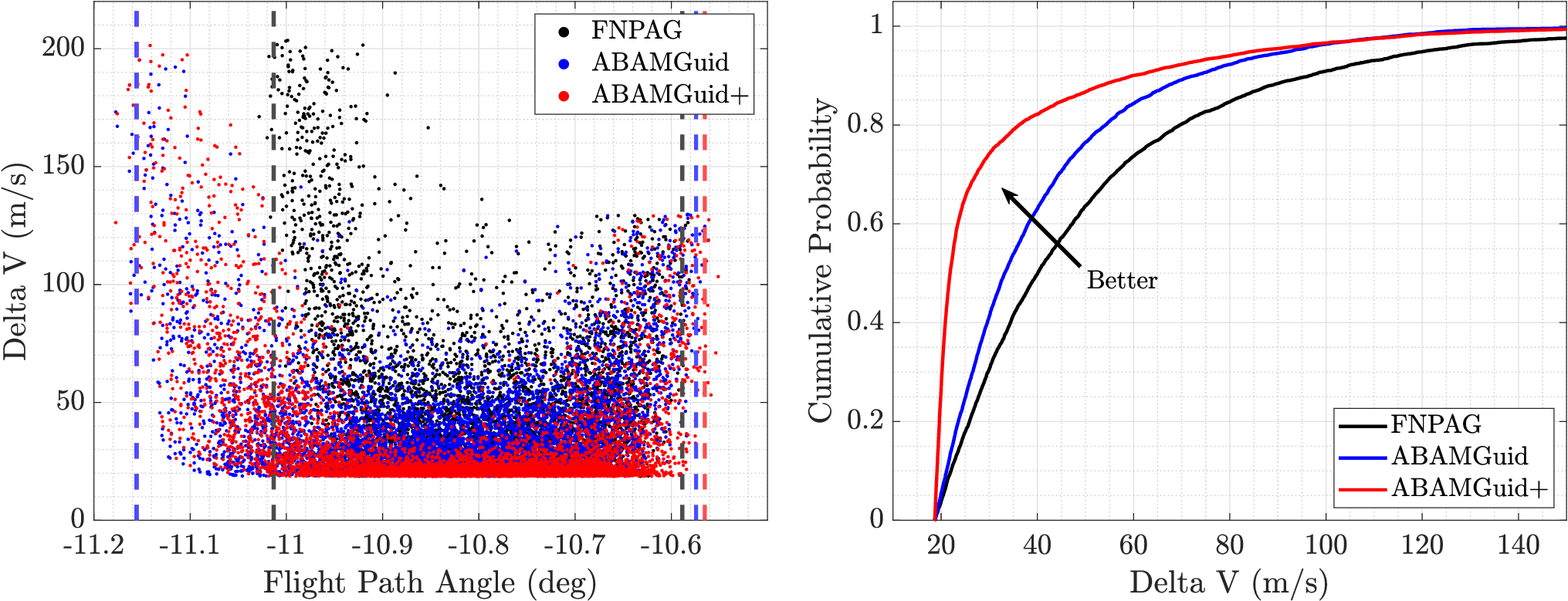}
	\caption{Monte Carlo $\Delta V$ results for the conservative entry state set (Test Case 2). A scatter plot of $\Delta V$ vs EFPA is shown in the left panel, and the $\Delta V$ CDF is shown in the right panel. In comparison to ABAMGuid (blue) and FNPAG (black), ABAMGuid+ (red) results in generally lower $\Delta V$ with higher probability.  Vertical dashed lines contain 99.7\% of successful runs, defining the entry corridor width for each algorithm. ABAMGuid+ results in a similar ECW as ABAMGuid, which are both significantly wider than FNPAG. Simulations which fail to be captured are omitted.} 
	\label{fig: conservative mc figs}
\end{figure}

Figure \ref{fig: apoapsis targeting} illustrates the achieved apoapsis radius of the post-aerocapture orbit for each algorithm in the baseline (right) and conservative (left) entry state sets under \textit{Test Case 2}. Although the algorithms are configured to target exit velocity, each algorithm is able to achieve a distribution centered around the apoapsis target radius. This confirms that the exit velocity reformulation is an effective approximation for this problem. For both entry sets, ABAMGuid+ achieves a tighter clustering around the desired apoapsis, indicating that ABAMGuid+ is most effective at delivering the vehicle into the specific orbit. While this section presents figures for only \textit{Test Case 2}, since they are representative of the broader trends, the complete figure set is included in \cite{sonandres2025thesis}.

\begin{figure}[tb!]
	\centering 
	\includegraphics[width = 0.7\textwidth]{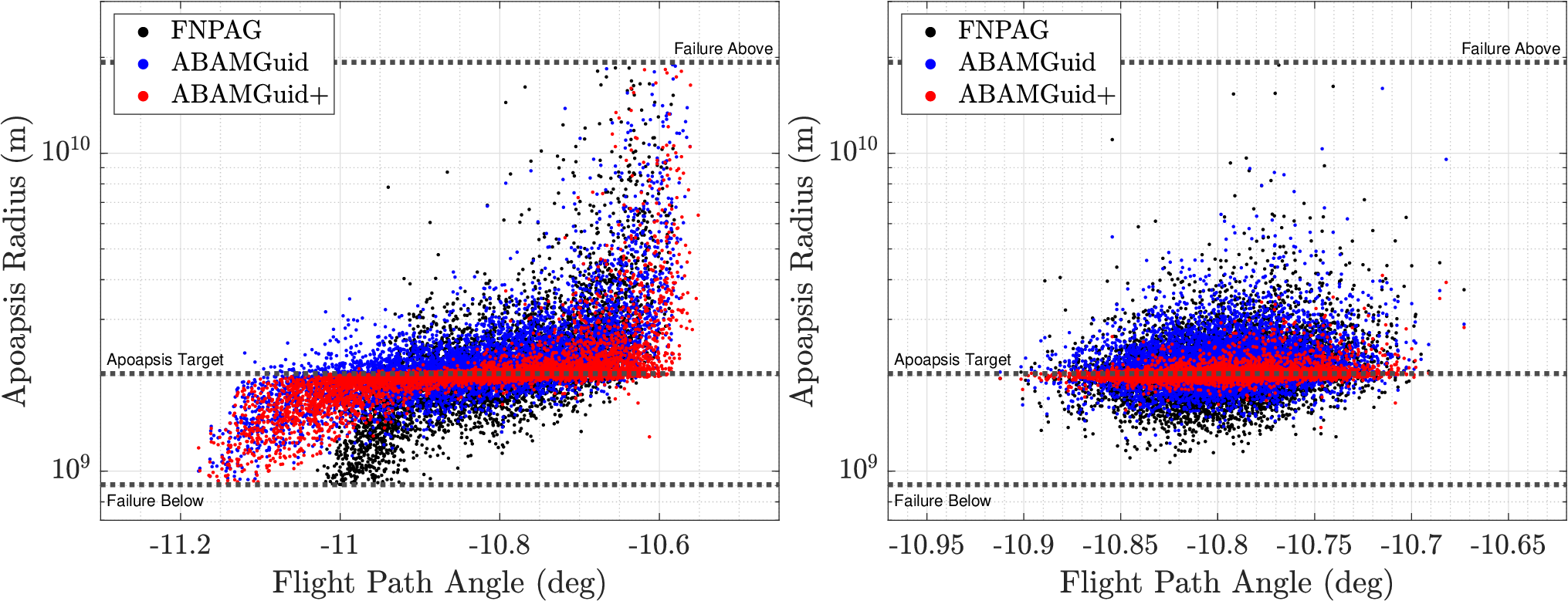}
	\caption{Monte Carlo apoapsis targeting results for the conservative (left panel) and baseline (right panel) entry sets (Test Case 2) ABAMGuid+ (red) achieves the tightest clustering around the target apoapsis radius across both entry state sets. Simulations which fail to be captured are omitted.} 
	\label{fig: apoapsis targeting}
\end{figure}

%% file: Conclusion.tex
In this paper we present ABAMGuid+, an aerocapture guidance algorithm for augmented bank angle modulation that improves upon our previous algorithm, ABAMGuid, via a new algorithm subprocess to perform unsaturated control of $\alpha$ and $\sigma$ (CASM). We derived propellant optimal control profiles for ABAM using both linear and quadratic aerodynamic models, the solutions of which both influenced the design of ABAMGuid+. The linear solution revealed a three-switch bang-bang solution, and the quadratic solution revealed an unsaturated $\alpha$ control law that motivated development of CASM. Both optimal control solutions were validated by solving the constrained nonlinear optimization problem in GPOPS. Last, ABAMGuid+ was rigorously tested through high-fidelity Monte Carlo Analysis, which showed $\Delta V$ and capture success gains over the previous methods. Future work should assess the impact of path constraints and lateral guidance, both of which are not considered in this work. 